\documentclass[12pt]{amsart}
\usepackage{amssymb}
\usepackage{url}
\usepackage{graphicx}
\usepackage{tikz-cd}
\usepackage{mathtools}

\usepackage[all]{xy}

\allowdisplaybreaks

\newtheorem{theorem}[equation]{Theorem}

\newtheorem{lemma}[equation]{Lemma}
\newtheorem{proposition}[equation]{Proposition}
\newtheorem{corollary}[equation]{Corollary}

\newtheorem*{theorem:derhamisomorphism}{Theorem~\ref{T:derhamisomorphism}}
\newtheorem*{theorem:characterization}{Theorem~\ref{T:characterization}}

\theoremstyle{definition}
\newtheorem{definition}[equation]{Definition}
\newtheorem{example}[equation]{Example}

\newtheorem{remark}[equation]{Remark}

\numberwithin{equation}{section}

\newcommand{\FF}{\mathbb{F}}
\newcommand{\ZZ}{\mathbb{Z}}

\newcommand{\TT}{\mathbb{T}}

\newcommand{\CC}{\mathbb{C}}

\newcommand{\zz}{\bold{z}}
\newcommand{\EE}{\bold{E}}

\DeclarePairedDelimiter\floor{\lfloor}{\rfloor}
\DeclareMathOperator{\Der}{Der}

\DeclareMathOperator{\Ker}{Ker}
\DeclareMathOperator{\GL}{GL}
\DeclareMathOperator{\Mat}{Mat}

\DeclareMathOperator{\End}{End}
\DeclareMathOperator{\Eis}{Eis}

\DeclareMathOperator{\Res}{Res}

\DeclareMathOperator{\Id}{Id}

\DeclareMathOperator{\trdeg}{tr.deg}

\newcommand{\dnorm}[1]{\lVert #1 \rVert}
\newcommand{\inorm}[1]{{\lvert #1 \rvert}}

\begin{document}

\title[On Drinfeld modular forms of higher rank and quasi-periodic functions]{On Drinfeld modular forms of higher rank and quasi-periodic functions}

\author{Yen-Tsung Chen}
\address{Department of Mathematics, National Tsing Hua University, Hsinchu City 30042, Taiwan
	R.O.C.}

\email{ytchen.math@gmail.com}

	\author{O\u{g}uz Gezm\.{i}\c{s}}
\address{National Center for Theoretical Sciences, National Taiwan University, Taipei, Taiwan R.O.C.}
\email{gezmis@ncts.ntu.edu.tw}



\date{August 21, 2021}

\thanks{The first author was partially supported by Professor Chieh-Yu Chang’s MOST Grant 107-2628-M-007-002-MY4.}

\thanks{The second author was supported by the MOST Grant 109-2811-M-007-553.}



\maketitle

\begin{abstract}
	
	In the present paper, we introduce a special function on the Drinfeld period domain $\Omega^{r}$ for $r\geq 2$ which gives the false Eisenstein series of Gekeler when $r=2$. We also study its functional equation and relation with quasi-periodic functions of a Drinfeld module as well as transcendence of its values at CM points.

	
\end{abstract}

\section{Introduction}

 Let $\mathbb{F}_q$ be the finite field with $q$ elements where $q$ is a positive power of a prime $p$ and let $\theta$ be an indeterminate over $\mathbb{F}_q$. We define $A=\mathbb{F}_q[\theta]$ and $A_{+}$ to be the set of monic polynomials of $A$. We let $K=\mathbb{F}_q(\theta)$ and  set $\inorm{\cdot}$ to be the $\infty$-adic norm normalized so that $\inorm{\theta}=q$. Let $K_{\infty}=\mathbb{F}_q((1/\theta))$ be the completion of $K$ with respect to $\inorm{\cdot}$ and $\CC_{\infty}$ be the completion of the algebraic closure of $K_{\infty}$. We also set $\overline{K}$ to be the algebraic closure of $K$ in $\CC_{\infty}$. 

In 1974, Drinfeld \cite{Dri74} introduced ``\textit{elliptic modules}'' nowadays called \textit{Drinfeld modules} as one of the tools for his solution to the Langlands conjectures for $\GL_2$ in the global function field case. He further defined the category of Drinfeld modules of rank $r\in \ZZ_{\geq 1}$ and $A$-lattices that are strongly discrete projective $A$-modules of rank $r$ and showed the equivalence of these two categories (see \S2.1 for details). 

For any integer $r\geq 2$, let $\Omega^r:=\mathbb{P}^{r-1}(\CC_{\infty})\setminus \{K_{\infty}\text{-rational hyperplanes}\}$ be the Drinfeld upper half plane which can be identified as the set of elements $\zz=(z_1,\dots,z_r)^{\intercal}$ whose entries are $K_{\infty}$-linearly independent and normalized so that $z_r=1$. Due  to the construction, entries of each point $\zz\in \Omega^r$ form an $A$-basis for a rank $r$ $A$-lattice $\zz A$ corresponding to a Drinfeld module, which we will denote by $\phi^{\zz}$, via the equivalence of categories (see \S2.1 for further details). 

Each Drinfeld module $\phi^{\zz}$ has the corresponding exponential function $\exp_{\zz A}:\CC_{\infty}\to \CC_{\infty}$ that is also entire. Furthermore for each $1\leq i \leq r-1$, there exists a unique entire function $F_i:\CC_{\infty}\to \CC_{\infty}$, which we call a \textit{quasi-periodic function}, satisfying the functional equation
\begin{equation}\label{E:qpfe}
F_i(\theta z)-\theta F_i(z)=\exp_{\zz A}(z)^{q^i}, \  \ z\in \CC_{\infty}.
\end{equation}
Moreover we call any non-zero element $z'$ of $\zz A$ a \textit{period} and $F_i(z')$ a \textit{quasi-period} of $\phi^{\zz}$. 
The theory of quasi-periodic functions are developed by Anderson, Deligne, Gekeler and Yu in the function field setting and we refer the reader to \S2.2 for further details.

Based on Drinfeld's seminal work \cite{Dri74}, Goss \cite{Gos80} introduced the theory of Drinfeld modular forms in the rank 2 case both algebraically and analytically, which can be also compared with the classical setting of modular forms. He also introduced conditions of holomorphy at infinity which allowed him to prove the finiteness of the  dimension of certain spaces of Drinfeld modular forms. 

Later on Gekeler made several contributions to the theory by revealing further the analogy with the classical setting   \cite{Gek84}, \cite{Gek88}, \cite{Gek89}. In \cite{Gek88}, he defined the positive characteristic analogue of the false Eisenstein series $E:\Omega^2\to \CC_{\infty}$ which is a rigid analytic function and he related the values of $E$ to the values of a quasi-periodic function of a Drinfeld module \cite[Thm. 7.10]{Gek89}. In 2008, Bosser and Pellarin \cite{BP08} defined quasi-modular forms by introducing the analogy with the classical setting and studied the differential algebra structure of them in a detailed way. We note that the function $E$ is not a Drinfeld modular form but in fact is a quasi-modular form \cite[Sec. 1.2]{BP08}.

In the 1980s Goss predicted a suitable generalization of the theory of Drinfeld modular forms to the higher rank case based on the work of Kapranov \cite{Kap88} and Gekeler \cite{Gek86} studied the compactification of Drinfeld modular varieties and introduced a parameter at infinity. Later on Basson, Breuer and Pink \cite{BBP18a} defined the analytic Drinfeld modular forms in arbitrary rank (see also \cite{B14}). In \cite{BBP18b}, they also prove the comparison isomorphism \cite[Thm. 10.9]{BBP18b} between the space of analytic Drinfeld modular forms and the space of algebraic Drinfeld modular forms defined by Pink \cite{Pin13} based on the Satake compactification of Drinfeld moduli spaces introduced by Kapranov \cite{Kap88} and later by H\"{a}berli \cite{Hab21} who generalized the work of Kapranov  (see \S2.3 for details). Through their analysis, they consider several examples of Drinfeld modular forms such as the non-vanishing \textit{discriminant function}  $\Delta_{r}:\Omega^r\to \CC_{\infty}$ whose product formula is discovered by Basson \cite[Thm. 8]{Bas17} and Hamahata \cite{Ham} by different methods (see \S2.3 for more details). We  emphasize that Gekeler contributed to the higher rank setting also in \cite{Gek17}, \cite{Gek18a}, \cite{Gek18b}, \cite{Gek18} and \cite{Gek19} as well as in his unpublished works (see  Example \ref{Ex:1}(iii)). For further information on the history of Drinfeld modular forms, we refer the reader to \cite[Sec. 7]{BB17} and \cite[Sec. 1.1]{Per14}.

The main goal of the present paper is to introduce a generalization of the function $E$ to the higher rank setting and obtain some of its properties by explaining their connection with the known results for the false Eisenstein series of Gekeler.  For this purpose, we consider $\Delta_{r}$ to be a function of variables $z_1,\dots,z_{r-1}$ and define $E_{r}:\Omega^r\to \CC_{\infty}$ by 
\begin{equation}\label{E:int1}
E_r(\zz):=\tilde{\pi}^{-1}\Delta_{r}(\zz)^{-1}\frac{\partial}{\partial z_1}\Delta_{r}(\zz)
\end{equation}
 where $\tilde{\pi}\in \CC_{\infty}^{\times}$ is  defined up to multiplication by any $(q-1)$-st root of unity (see \S2.1  for the explicit definition) and $\partial/\partial z_1$ is the partial derivative with respect to $z_1$.  
 
  Our first result, which will be restated as Theorem \ref{T:EGEK} later, is as follows.
 \begin{theorem}\label{T:1}
 	Let $\zz=(z_1,\dots,z_r)^{\intercal}$ be an element in $\Omega^r$.  The function $E_r$ is rigid analytic on $\Omega^r$. Furthermore, for each $1\leq i \leq r-1$, let $F_i:\CC_{\infty}\to \CC_{\infty}$ be the quasi-periodic function satisfying \eqref{E:qpfe}. Then we have
 		\[
 		E_r(\zz)=\tilde{\pi}^{-1+q+\dots+q^{r-1}}h_r(\zz)\det\begin{pmatrix}
 		F_{1}(z_2)&\dots &F_{r-1}(z_2)\\
 		\vdots & & \vdots \\
 		F_{1}(z_r)&\dots &F_{r-1}(z_r)
 		\end{pmatrix} 
 		\]
 		where $h_r: \Omega^r\to \CC_{\infty}$ is the Drinfeld modular form introduced by Gekeler and defined as an explicit multiple of a fixed $(q-1)$-st root of the discriminant function.
\end{theorem}
Theorem \ref{T:1} implies that $E_r$ is not a Drinfeld modular form (see \S 5 for details). We also note that $E_2$ is the function $E$ \cite[Sec. 8]{Gek88} and Theorem \ref{T:1} is proved by Gekeler \cite[Thm. 7.10]{Gek89} for the rank two setting by using a detailed analysis on quasi-periodic functions and the function field analogue of the valence formula \cite[(5.14)]{Gek88}. For the higher rank case, we instead follow a different argument to prove Theorem \ref{T:1} motivated by Pellarin's work \cite[Sec. 9]{Pel19}  on the Eisenstein series with values in the Tate algebra and their specialization at explicit points of $A$ (see \S3 for details).

 We call $\zz \in \Omega^r$ a \textit{CM point} if the endomorphism ring $\End(\phi^{\zz})$ of $\phi^{\zz}$ is a free $A$-module of rank $r$ (see \S2.1 for details). Several results concerning the values of Drinfeld modular forms and quasi-modular forms at CM points are obtained by Chang in \cite{Cha11} and \cite{Cha13} for the rank 2 case. In particular, the transcendence  over $\overline{K}$ of the values of $E=E_2$ at CM points is shown in \cite[Thm. 1.2]{Cha13}. 
 
 We have the following theorem on the special values of our function $E_r$. 
\begin{theorem} \label{T:2}
		Let $\mathbf{z}\in\Omega^r$ be a CM point. If $E_r(\mathbf{z})\neq 0$, then it is transcendental over $\overline{K}$.
\end{theorem}
The outline of the paper and the strategy of the proof of Theorem \ref{T:1} and Theorem \ref{T:2} can be described as follows: In \S2, we introduce the necessary notation and background which will be in use throughout the paper. In \S3,  we obtain a formula (Theorem \ref{T:main}) relating the Eisenstein series over the Tate algebra   to the Anderson generating functions by using Pellarin's method described in \cite[Sec. 8]{Pel19} and an analysis on the exponential function of a Drinfeld module. Then in \S4, we discuss some properties of the function $E_r$ such as its $u$-expansion and functional equation which will imply that $E_r$ is a rigid analytic function. In \S5, we obtain Theorem \ref{T:1} by using Theorem \ref{T:main}. Finally in \S6, we give a proof of Theorem \ref{T:2} (stated as Theorem \ref{T:trans} later) by using Theorem \ref{T:1} and an algebraic independence result of Chang and Papanikolas \cite[1.2.2]{CP12} on periods and quasi-periods of Drinfeld modules. 

\begin{remark} Inspired by the results in Theorem \ref{T:1} and Theorem \ref{T:2}, it is natural to consider the function $E_r$ to be \textit{the false Eisenstein series} in the higher rank case. Therefore it would be also interesting to understand and define  quasi-modular forms precisely in the arbitrary rank setting and study their differential algebra structure as was done by Bosser and Pellarin \cite{BP08} in the rank two setting. We hope to tackle this problem in the near future.

\end{remark}

\subsection*{Acknowledgments} The authors are grateful to Chieh-Yu Chang, Andreas Maurischat, Federico Pellarin, Fu-Tsun Wei and Jing Yu for useful suggestions and fruitful discussions. The authors also thank the anonymous referees for many useful comments and suggestions to make the presentation of the results clearer.

\section{Preliminaries and Background}
\subsection{Drinfeld modules and Anderson generating functions} Let $R$ be an $A$-algebra in $\CC_{\infty}$ and $\tau : R \to R$ be the Frobenius map defined by $\tau(x)=x^q$ for all $x\in R$. We define the non-commutative skew polynomial ring $R[\tau ]$  and the formal power series ring $R[[\tau]]$ subject to the relation $\tau x=x^q\tau$ for all $x\in R$. 

Let $r$ be a positive integer. \textit{A Drinfeld module $\phi$ of rank $r$ defined over $R$} is an $\mathbb{F}_q$-algebra homomorphism 
$
\phi \colon A \to R[\tau]
$
given by 
\begin{equation}\label{intro0}
\theta \mapsto \phi_{\theta}:=\theta + \phi_{\theta,1}\tau + \dots + \phi_{\theta,r}\tau^r, \ \  \phi_{\theta,r} \neq 0.
\end{equation}
For each $1\leq i \leq r$, we call $\phi_{\theta,i}$ \textit{the $i$-th coefficient of $\phi$}. We say that the Drinfeld modules $\phi$ and $\varphi$ are \textit{isogenous} if there exists a non-zero element $v\in \CC_{\infty}[\tau]$ such that $v\phi_{\theta}=\varphi_{\theta}v$. We call $v$ an \textit{isogeny} between $\phi$ and $\varphi$. These elements form the morphisms of the category of Drinfeld modules and it can be easily shown that  non-trivial morphisms  occur only between Drinfeld modules of the same rank. Moreover we call $\phi$ and $\varphi$ are \textit{isomorphic} if $v$ is in $\CC_{\infty}^{\times}$.

We further define the endomorphism ring $\End(\phi)$ of $\phi$ to be the subring of $\CC_{\infty}[\tau]$ given by 
$
\End(\phi):=\{v\in \CC_{\infty}[\tau] \ | \ v \phi_{\theta}=\phi_{\theta}v    \}.
$
It forms a commutative ring which is also a free and finitely generated  $A$-module whose rank is less than or equal to $r$ \cite{Dri74}. We say that $\phi$ has \textit{complex multiplication} (CM) if the dimension of $\End(\phi)\otimes_{A} K$ over $K$ is $r$ and furthermore we call $\phi$ a \textit{CM Drinfeld module}.

Letting $c\in \CC_{\infty}^{\times}$, we consider the open ball 
$
\mathbb{D}_\inorm{c}:=\{\alpha\in \CC_{\infty} \ \ | \ \ \inorm{\alpha}<\inorm{c}\}
$ and say that an $A$-module $M$ is \textit{strongly discrete} if the intersection of $M$ with $\mathbb{D}_\inorm{c}$ is finite for any $c\in \CC_{\infty}$. Furthermore we call $M$ \textit{an $A$-lattice} if it is free, finitely generated and strongly discrete. We call two $A$-lattices $\Lambda$ and $\Lambda'$ \textit{isogenous} if there exists an element $c\in \CC_{\infty}^{\times}$ such that $c \Lambda \subset \Lambda'$ where the quotient $\Lambda'/c\Lambda$ is a finite $A$-module. We furthermore call $c$ an \textit{isogeny between $\Lambda$ and $\Lambda'$} and say that $\Lambda$ and $\Lambda'$ are \textit{isomorphic} if $c\Lambda=\Lambda'$. By \cite{Dri74}, there is an equivalence between the category of Drinfeld modules of rank $r$ and the category of $A$-lattices of rank $r$.

Let $\Lambda$ be the $A$-lattice of rank $r$ corresponding to $\phi$. \textit{The exponential series}
$
\exp_{\phi}=\sum_{j\geq 0}\alpha_j\tau^j \in \CC_{\infty}[[ \tau ]]
$
is uniquely  defined so that $\alpha_0=1$ and $\exp_{\phi}a=\phi_{a}\exp_{\phi}$ for all $a\in A$ . It induces an entire function $\exp_{\phi} \colon \CC_{\infty} \to \CC_{\infty}$ defined by 
\[
\exp_{\phi}(x)=x\prod_{\substack{\lambda\in \Lambda\\\lambda\neq 0}}\Big(1-\frac{x}{\lambda}\Big)=\sum_{j\geq 0}\alpha_jx^{q^j}, \ \ x\in \CC_{\infty}.
\]
It is clear from the definition that the $A$-lattice $\Lambda$ is the kernel $\Ker(\exp_{\phi})$ of $\exp_{\phi}$. An example of Drinfeld modules is the Carlitz module $C$ defined by 
$
C_{\theta}=\theta+\tau
$
and $\Ker(\exp_{C})$ is the $A$-lattice of rank one generated by the element $\tilde{\pi}\in \CC_{\infty}^{\times}$ given as 
\[
\tilde{\pi}:=\theta(-\theta)^{1/(q-1)}\prod_{i=1}^{\infty}\Big(1-\theta^{1-q^i}\Big)^{-1}
\]
where $(-\theta)^{1/(q-1)}$ is a fixed choice of $(q-1)$-st root of $-\theta$. 

Considering $t$ as a variable over $\CC_{\infty}$, we define \textit{the Tate algebra} $\TT$ by 
\[
\TT:=\Big\{\sum_{i\geq 0}c_it^i \  \ | \  \ c_i\in \CC_{\infty}, \ \ \inorm{c_i}\to 0 \text{ as } i\to \infty      \Big\}.
\]
One can equip $\TT$ with a Banach algebra structure by defining the Gauss norm $\dnorm{g}$ for any $g=\sum_{i\geq 0}c_it^i\in \TT$ by 
\[
\dnorm{g}:=\sup\{\inorm{c_i} \ \ |i\in \ZZ_{\geq 0}\}.
\]
We also extend the norm $\dnorm{\cdot}$ to  $\Mat_{n\times m}(\TT)$ by simply for any $M=(g_{ik})\in \Mat_{n\times m}(\TT)$ setting $\dnorm{M}:=\sup\{\dnorm{g_{ik}}\}$.

Note that the elements of $\TT$ can be also considered as rigid analytic functions of $t$, in the sense of \cite[Sec. 2.2]{FvdP04}, converging in the closed unit disk. Hence it is important to point out that throughout the paper we will specifically mention when we evaluate functions converging on the closed unit disk or even on a larger domain, that is when $t$ stands for a value in $\CC_{\infty}$ lying in the convergence domain of such function.

For any $g=\sum_{i\geq 0} c_it^i\in \TT$ and $j\in \mathbb{Z}$, we set $g^{(j)}:=\sum_{i\geq 0} c_i^{q^j}t^i$. Similarly for any $M=(g_{ik})\in \Mat_{n\times m}(\TT)$, we define $M^{(j)}:=(g_{ik}^{(j)})\in\Mat_{n\times m}(\TT) $. Furthermore, we define $\TT\{t/c\}$ to be the subset of $\TT$ consisting of elements $g=\sum_{i\geq 0}c_it^i\in \TT$ which are convergent on $\mathbb{D}_{\inorm{c}}$ as a function of $t$. 

Let $z$ be an element in $\CC_{\infty}$. We define \textit{the Anderson generating function} $s_{\phi}(z;t)$ by 
\begin{equation}\label{E:AndGF}
s_{\phi}(z;t):=\sum_{i\geq 0}\exp_{\phi}\Big(\frac{z}{\theta^{i+1}}\Big)t^i\in \CC_{\infty}[[t]].
\end{equation}

For any element $a\in A$, we set $a(t):=a_{|\theta=t}$ where $t$ is either a variable or a value in $\CC_{\infty}$. We obtain the following properties of $s_{\phi}(z;t)$ in the next proposition  whose proof is due to Pellarin (see also \cite[Prop. 3.2]{EGP14}).
\begin{proposition}\cite[Sec. 4.2]{Pel08}  \label{P:Anderson} Let $z$ be an element in $\CC_{\infty}$. Consider the exponential series $
	\exp_{\phi}=\sum_{j\geq 0}\alpha_j\tau^j \in \CC_{\infty}[[ \tau ]]
	$ of $\phi$.
	\begin{itemize}
	\item[(i)] We have an identity which holds in $\TT$:
	\[
	s_{\phi}(z;t)= \sum_{j=0}^{\infty}\frac{\alpha_j z^{q^j}}{\theta^{q^j}-t}\in \TT.
	\]
	\item[(ii)] As a function of $t$, $s_{\phi}(z;t)$ has poles at the points $t= \theta^{q^j}$ for $j\in \ZZ_{\geq 0}$ with residues $\Res_{t=\theta^{q^j}}s_{\phi}(z;t) = -\alpha_j z^{q^j}$. 
	\item[(iii)] For any $n\geqslant 1$, $s_{\phi}^{(n)}(z;t)$ lies in $\TT\{t/\theta^{q^{n-1}}\}$.
	\item[(iv)] Consider the polynomial $\psi_\phi = \phi_{\theta,r} \tau^r + \dots + \phi_{\theta,1} \tau - (t - \theta) \in \CC_{\infty}[\tau]$. For any $z\in \Ker(\exp_{\phi})$, we have $\psi_{\phi}(s_{\phi}(z;t)) = 0$.
	\item[(v)] For any $a\in A$ and $z\in \Ker(\exp_{\phi})$, we have the equality
	$
	s_{\phi}(az;t)=a(t)s_{\phi}(z;t)
	$ in $\TT$.
	\end{itemize}
	
\end{proposition}
\subsection{Quasi-periodic functions} In this section we introduce the theory of quasi-periodic functions in the function field setting (see also \cite{Gekeler89} and \cite{Yu90}). Let $\phi$ be a Drinfeld module of rank $r$. The $\mathbb{F}_q$-linear map $\eta:A\to \tau\CC_{\infty}[\tau]$ is called \textit{a biderivation} if it satisfies
\[
\eta_{ab}=a\eta_b+\eta_a\phi_b
\]
for all $a,b\in A$. Furthermore, we call a biderivation $\eta^{(m)}:A\to \tau \CC_{\infty}[\tau]$ \textit{ inner} if it is given by $\eta_a^{(m)}=m\phi_a-am$ for some $m\in \CC_{\infty}[\tau]$. If $m\in \tau\CC_{\infty}[\tau]$, then we say that $\eta^{(m)}$ is \textit{strictly inner}. For later use, we set the notation $\delta_0$ for the biderivation $\eta^{(-1)}$ so that it is given by $\delta_0:a\mapsto a-\phi_a$. We let $\Der(\phi)$  to be the set of all biderivations of $\phi$. It has a left $\CC_{\infty}$-module structure given by $(c\cdot \eta)_{\theta}=c\eta_{\theta}$ for $c\in \CC_{\infty}$ and $\eta\in \Der(\phi)$. Let $\Der_{si}(\phi)$ be the set of all strictly inner biderivations  and define the \textit{de Rham module} $H_{DR}^{*}(\phi)$ by $H_{DR}^{*}(\phi):=\Der(\phi)/\Der_{si}(\phi)$. For each $1\leq i \leq r-1$, we set the biderivation $\delta_i\in \Der(\phi)$ given by $\delta_i:\theta\mapsto \tau^i$. By \cite{Gekeler89}, we know that $\delta_0,\dots,\delta_{r-1}$ forms a left $\CC_{\infty}$-module basis for $H_{DR}^{*}(\phi)$. 

Now let $\eta\in \Der(\phi)$. By \cite[(2.2)]{Gekeler89}, we know that there exists a unique entire function $F_{\eta}^{\phi}:\CC_{\infty}\to \CC_{\infty}$, which we call \textit{the quasi-periodic function corresponding to $\eta$}, satisfying the functional equation
\[
F_{\eta}^{\phi}(\theta z)-\theta F_{\eta}^{\phi}(z)=\eta_{\theta}(\exp_{\phi}(z)),  \ \ z\in \CC_{\infty}.
\] 
As an example, the quasi-periodic function $F_{\delta_0}^{\phi}$ corresponding to the biderivation $\delta_0$ is given by $F_{\delta_0}^{\phi}(z)=z-\exp_{\phi}(z)$. 
Let $w_1,\dots,w_r$ be the generators of the $A$-lattice $\Lambda=\Ker(\exp_{\phi})$ corresponding to $\phi$. We call each $w\in \Lambda$ \textit{a period} of $\phi$ and $F^{\phi}_{\delta_j}(w)$  \textit{a quasi-period of $\phi$} for any $1\leq j \leq r-1$.  We also set \textit{the period matrix of $\phi$} to be  
\begin{equation}\label{E:period}
P:=
\begin{pmatrix}
w_1& F^{\phi}_{\delta_1}(w_1) &\dots &F^{\phi}_{\delta_{r-1}}(w_1)\\
\vdots &\vdots & & \vdots \\
w_r& F^{\phi}_{\delta_{1}}(w_r) &\dots &F^{\phi}_{\delta_{r-1}}(w_r)
\end{pmatrix}\in \Mat_{r}(\CC_{\infty}).
\end{equation}

The next proposition  reveals a relation between quasi-periodic functions and Anderson generating functions.
\begin{proposition}\cite[Sec. 4.2]{Pel08}\label{P:quasiper} Let $w$ be an element in $\Lambda$ and $s_{\phi}(w;t)$ be the Anderson generating function as in \eqref{E:AndGF}. Then for each $1\leq i \leq r-1$, we have $F^{\phi}_{\delta_i}(w)=s_{\phi}^{(i)}(w;t)_{|t=\theta}$.
	
\end{proposition}

\subsection{Drinfeld modular forms} Recall that for any integer $r\geq 2$, we introduce the Drinfeld upper half plane $\Omega^r:=\mathbb{P}^{r-1}(\CC_{\infty})\setminus \{K_{\infty}\text{-rational hyperplanes}\}$ which can be identified as the set of elements $\bold{z}=(z_1,\dots,z_r)^{\intercal}$ consisting of $K_{\infty}$-linearly independent entries and normalized so that $z_r=1$. By \cite{Put87}, it is a rigid analytic space which is simply connected. We also set $\tilde{\zz}=(z_2,\dots,z_r)^{\intercal}$. By abuse of notation, for any $c\in \CC_{\infty}^{\times}$, let $c\zz A$ ($c\tilde{\zz} A$ respectively) be the $A$-lattice generated by the entries of $c\zz$ ($c\tilde{\zz}$  respectively) which are, by definition, $K_{\infty}$-linearly independent. We let $\phi^{\bold{z}}$ to be the Drinfeld module of rank $r$ corresponding to  $\zz A$  defined by 
\begin{equation}\label{E:drinfeld2}
\phi^{\bold{z}}_{\theta}:=\theta+g_1(\bold{z})\tau+\dots +g_{r-1}(\bold{z})\tau^{r-1}+\Delta_r(\bold{z})\tau^r.
\end{equation}

Let $\gamma=(a_{ij})\in \GL_r(A)$. We define the action of $\GL_r(A)$ on $\Omega^r$ by
\begin{equation}\label{E:action}
\gamma \cdot \zz:=\Big(\frac{a_{11}z_1+\dots+ a_{1r}z_r}{a_{r1}z_1+\dots+ a_{rr}z_r},\dots,\frac{a_{(r-1)1}z_1+\dots +a_{(r-1)r}z_r}{a_{r1}z_1+\dots+ a_{rr}z_r},1\Big)^{\intercal}\in \Omega^r
\end{equation}
and set $j(\gamma,\zz):=a_{r1}z_1+\dots+ a_{rr}z_r$ so that $j(\gamma,\zz)$ satisfies certain properties of being a factor of automorphy \cite[Lem. 3.1.3]{B14}. We define  
$
u(\zz):=\exp_{\tilde{\pi}\tilde{\zz}A}(\tilde{\pi}z_1)^{-1}\in \CC_{\infty}^{\times}
$
where $\exp_{\tilde{\pi}\tilde{\zz}A}$ is the exponential function of the Drinfeld module corresponding to the $A$-lattice $\tilde{\pi}\tilde{\zz}A$.
We call a rigid analytic function $f:\Omega^r\to \CC_{\infty}$ \textit{a weak modular form (for $\GL_r(A)$) of weight $k\in \ZZ$ and type $m\in \ZZ/(q-1)\ZZ$} if it satisfies
\[
f(\gamma\cdot \zz)=j(\gamma,\zz)^k\det(\gamma)^{-m}f(\zz),\ \  \gamma\in \GL_r(A), \ \ \zz \in \Omega^r.
\]
For $\zz\in\Omega^r$, we define the function 
\[
|\zz|_i:=\mathrm{inf}\{\inorm{z_1-\alpha}:\alpha=a_2z_2+\dots+a_rz_r, \ \  a_2,\dots,a_r\in K_{\infty}\}.
\]
 Consider the map 
$
\iota:A^{r-1}\to \GL_r(A)$ given by
\[
\iota:(a_2,\dots,a_r)\mapsto\begin{pmatrix}
1 & a_2 & \cdots &a_r\\
 & 1 & &\\
 & & \ddots &\\
 & & & 1
\end{pmatrix}.
\]
Then a function $f:\Omega^r\to\mathbb{C}_\infty$ is called \textit{$\tilde{\zz}A$-invariant} if  $f(\gamma\cdot \zz)=f(\zz)$ for all $\gamma\in \iota(A^{r-1})$, and it is moreover \emph{holomorphic at infinity} if for any $\Tilde{\zz}\in\Omega^{r-1}$, 
we have (cf. \cite[Prop. 3.2.9]{B14} and \cite[Def. 5.12]{BBP18a})
\[
\lim_{|z_1|=|\zz|_i\to\infty}\inorm{f(\zz)}<\infty,~\zz=(z_1,\Tilde{\zz})\in\Omega^r.
\]

By \cite[Prop. 5.4]{BBP18a} we know that for any $\tilde{\zz}A$-invariant rigid analytic function $f:\Omega^r\to \CC_{\infty}$, there exists a uniquely defined holomorphic function $f_n:\Omega^{r-1}\to \CC_{\infty}$ for each $n\in \ZZ$ such that the series
\[
\sum_{n\in \ZZ}f_n(\tilde{\zz})u(\zz)^n
\]
converges to $f(\zz)$ on some neighborhood of infinity and its admissible subsets (see \cite[Def. 4.12]{BBP18a} for the details on a neighborhood of infinity). Such an expansion is called \textit{the $u$-expansion of $f$}. Note that when $r=2$,  $f_n\in\CC_{\infty}$ becomes a constant function for each $n$. Since any $\tilde{\zz}A$-invariant rigid analytic function on $\Omega^r$ can be determined by an expansion as above due to the connectedness of $\Omega^r$, we use a slight abuse of notation and simply write
\[
f(\zz)=\sum_{n\in \ZZ}f_n(\tilde{\zz})u(\zz)^n.
\]

It is easy to show that every weak modular form $f$ of weight $k$ and type $m$ is $\tilde{\zz}A$-invariant and we call $f$ \textit{a Drinfeld modular form (for $\GL_r(A)$) of weight $k$ and type $m$} if $f$ is also holomorphic at infinity. Note that it is equivalent to saying that the function $f_n$ in the $u$-expansion of $f$ is identically zero when $n<0$ by using \cite[Thm. 4.16]{BBP18a}.

\begin{example}\label{Ex:1}
	\begin{itemize}
		\item[(i)] Let $k\in \ZZ_{\geq 1}$ and $\zz=(z_1,\dots,z_r)^{\intercal}\in \Omega^r$. We define \textit{the Eisenstein series} $\Eis_{k}(\zz)$ by
		\[
		\Eis_k(\zz):= \sum{\vphantom{\sum}}'_{a_1,\dots,a_r\in A}\frac{1}{(a_1z_1+\dots+a_rz_r)^k}.
		\]
		where the dash $'$ on the summation denotes that we exclude the case $a_j=0$ for each $1\leq j \leq r$. It is shown in \cite{B14} that $\Eis_k(\zz)$ is a Drinfeld modular form of weight $k$ and type 0 (see also \cite{BB17} and \cite{BBP18}). Furthermore we set 
		\begin{equation}\label{E:expofz}
		\exp_{\phi^{\zz}}:=\sum_{i\geq 0}\alpha_i(\zz)\tau^i\in \CC_{\infty}[[\tau]]
		\end{equation}
		to be the exponential series of $\phi^{\zz}$. We have by \cite[Ex. 3.2]{BB17} that $\alpha_i:\Omega^r\to \CC_{\infty}$ is a Drinfeld modular form of weight $q^i-1$ and type 0 for each $i\in \ZZ_{\geq 0}$.
		\item[(ii)] For each $1\leq i \leq r-1$, the $i$-th coefficient $g_i(\zz)$ of $\phi^{\zz}$ is a Drinfeld modular form of weight $q^i-1$ and type 0 \cite[Prop. 3.4.4]{B14}. Similarly, \textit{the discriminant function} $\Delta_{r}:\Omega^r\to \CC_{\infty}$, determined by the $r$-th coefficient of $\phi^{\zz}$, is a Drinfeld modular form of weight $q^r-1$ and type 0. We also note the following non-trivial relation among the certain Drinfeld modular forms we have defined so far  (see \cite[(3.4)]{BB17}):
	\begin{equation}\label{E:formula}
	(\theta^{q^j}-\theta)\alpha_j(\zz)=g_j(\zz)+\sum_{k=1}^{j-1}g_{k}(\zz)\alpha_{j-k}(\zz)^{q^k},  \ \ 1\leq j\leq r-1.
	\end{equation}
	
		\item[(iii)] It is also possible to construct a Drinfeld modular form of non-zero type for each rank $r$. Let $S$ be a set of representatives of the quotient space $((A/\theta A)^r\setminus \{0\})/\mathbb{F}_q^{\times}$ given by 
		\[
		S:=\theta^{-1}\{(0,\dots,0,1),(0,\dots,0,1,*),\dots,(0,1,*,\dots,*),(1,*,\dots,*)\}.
		\]
		 For any $\mu=(\mu_1,\dots,\mu_r) \in S$, we define 
		\[
		\Eis_{\mu}(\zz):=\sum_{a_1,\dots,a_r\in A}\frac{1}{(a_1+\mu_1)z_1+\dots+(a_r+\mu_r)z_r}.
		\]
		We further set $\beta:=(-\theta)^{1/(q-1)}$  and introduce \textit{the $h$-function of Gekeler}  for rank $r$ by
		\[
		h_r(\zz):=-\beta\prod_{\mu \in S}\Eis_{\mu}(\zz).
		\]
	We note that the definition of $h_r$ is due to Gekeler introduced in his unpublished work which later appears in \cite{Gek17}. Moreover it is a non-vanishing Drinfeld modular form of weight $q^r-1/(q-1)$ and type 1 \cite{BB17} (see also \cite{BBP18}). Furthermore by \cite[Thm. 5.3(2)]{BB17}, we have
	\begin{equation}\label{E:det6}
	\Delta_r(\zz)=\tilde{\pi}^{q^r-1}(-1)^{r-1}h_r(\zz)^{q-1}.
	\end{equation}
	\end{itemize}
\end{example}

\section{Eisenstein series with values in the Tate algebra}
In this section, our goal is to study $\TT$-valued Eisenstein series (see \cite[Sec. 4]{Per14} for further details) and obtain formulas (Theorem \ref{T:main}) relating them to Anderson generating functions and the $h$-function of Gekeler.

Let $\zz=(z_1,\dots,z_r)\in \Omega^r$ and $\phi^{\bold{z}}$ be the Drinfeld module of rank $r$ corresponding to the $A$-lattice $\zz A$  defined as in \eqref{E:drinfeld2}. For each $1\leq i \leq r$, we set $s_i(\zz;t)$ to be the Anderson generating function given by 
$
s_i(\zz;t):=s_{\phi^{\zz}}(z_i;t).
$

We define the matrix
\begin{equation}\label{D:F}
\mathcal{F}(\zz,t):=\begin{pmatrix}
s_1(\zz;t)&\dots &s_1^{(r-1)}(\zz;t)\\
\vdots & & \vdots \\
s_r(\zz;t)&\dots &s_r^{(r-1)}(\zz;t)
\end{pmatrix}\in \Mat_{r}(\TT).
\end{equation}

For each $\bold{z}\in \Omega^r$ and $k\in \ZZ_{\geq 1}$, we consider  \textit{the Eisenstein series $\mathcal{E}_k(\bold{z},t)$ of weight $k$} which converges in $\Mat_{r\times 1}(\TT)$ with respect to Gauss norm $\dnorm{\cdot}$ and is given by 
\[
\mathcal{E}_\zz(k,t):=\begin{pmatrix}
\sum'_{a_1,\dots,a_r\in A}\frac{a_1(t)}{(a_1z_1+\dots+a_rz_r)^k},\dots,
\sum'_{a_1,\dots,a_r\in A}\frac{a_r(t)}{(a_1z_1+\dots+a_rz_r)^k}
\end{pmatrix}^{\intercal}.
\] 
A simple counting argument implies that $\mathcal{E}_\zz(k,t)$ is identically zero if $k\not\equiv 1\pmod{q-1}$. We also introduce the Anderson-Thakur function $\omega(t)$ defined by 
\[
\omega(t):=(-\theta)^{1/(q-1)}\prod_{i=0}^{\infty}\Big(1-\frac{t}{\theta^{q^i}}\Big)^{-1}\in \TT^{\times}.
\]
By \cite[(1.1)]{EGP14}, we know that it is the Anderson generating function $s_{C}(z;t)$ of the Carlitz module evaluated at $z=\tilde{\pi}$.

We are now ready to state our next result whose proof will occupy the rest of the present section.
\begin{theorem}\label{T:main} For the values of $\bold{z}=(z_1,\dots,z_r)^{\intercal}\in\Omega^r$ and $t\in \CC_{\infty}$ with $\inorm{t}\leq 1$, we have 
	\[
\mathcal{E}_\zz(1,t)^{\intercal}=\frac{\tilde{\pi}^{\frac{q^r-1}{q-1}}h_r(\zz)}{(t-\theta)\omega(t)}(C_{11},\dots,C_{r1})
	\]
	where  $C_{i1}$ is the $(i,1)$-cofactor of $\mathcal{F}(\zz,t)$ for $1\leq i \leq r$.
\end{theorem}

When $r=2$, Theorem \ref{T:main} is stated in several different ways by Pellarin in \cite{Pel11}, \cite[Thm. 8]{Pel12}, \cite[Thm. 9.9]{Pel19} and by Pellarin and Perkins in \cite[Cor. 3.14]{PP18}. For our proof in the higher rank case, we  apply an argument, motivated by \cite[Sec. 9]{Pel19}, which includes an analysis on the limiting behavior of the product of the entries of $\mathcal{E}_{\zz}(1,t)$ with Anderson generating functions.
\begin{remark}
We remark that in an ongoing work, Maurischat and Perkins also obtain formulas concerning the entries of
$\mathcal{E}_{\zz}(k,t)$ when $k$ is any non-negative power of $q$, which are used for further study on the $\TT$-valued Drinfeld modular forms introduced in \cite{Pel12} for the rank two setting and in \cite{Per14} for the higher rank setting.
\end{remark}
 
\subsection{Invertibility of $\mathcal{F}(\zz,t)$}
Our first goal is to show that the matrix $\mathcal{F}(\zz,t)$ is invertible by  calculating its determinant in terms of $h_{r}(\bold{z})$ and the Anderson-Thakur function $\omega(t)$.
\begin{proposition}\label{P:det} We have 
	$
	\det(\mathcal{F}(\zz,t))=\tilde{\pi}^{-\frac{q^r-1}{q-1}}h_r(\bold{z})^{-1}\omega(t).
$
\end{proposition}
\begin{proof} The proof uses  ideas of Anderson to obtain the Legendre relation and Pellarin in the proof of \cite[Lem. 2.3]{Pel14}. Let $\bold{z}=(z_1,\dots,z_r)^{\intercal}\in \Omega^r$  and consider the Drinfeld module $\phi^{\bold{z}}$ given as in \eqref{E:drinfeld2}.
	By Proposition \ref{P:Anderson}(iv), we have
	\begin{equation}\label{E:det1}
	g_1(\zz)s_i^{(1)}(\zz;t)+\dots  + g_{r-1}(\zz)s_i^{(r-1)}(\zz;t)+\Delta_r(\zz)s_i^{(r)}(\zz;t)=(t-\theta)s_i(\zz;t)
	\end{equation}
	for each $1\leq i \leq r$. Using \eqref{E:det1}, one can get the equality
	\begin{equation}\label{E:det2}
	\tau(\mathcal{F}(\zz,t)^{\intercal})=\Phi_{\zz}\mathcal{F}(\zz,t)^{\intercal}
	\end{equation}
	where $\Phi_{\zz}\in \Mat_r(\TT)$ is given by
	\[
	\Phi_{\zz}:=\begin{pmatrix}
	0&1& & & \\
	& &\ddots & &\\
	& & &\ddots & \\
	& & & & 1\\
	\frac{t-\theta}{\Delta_r(\bold{z})}& -\frac{g_1(\bold{z})}{\Delta_r(\bold{z})}&\dots & \dots & -\frac{g_{r-1}(\bold{z})}{\Delta_r(\bold{z})}
	\end{pmatrix}.
	\]
	For any fixed $\zz\in \Omega^r$, by Proposition \ref{P:Anderson}(iii), we have that $s_i^{(j)}(\zz;t)\in \TT\{t/\theta^{q^{j-1}}\}$ for  $j\in \ZZ_{\geq 0}$. Thus $\det(\mathcal{F}(\zz,t))\in \TT$. Taking the determinants of both sides of \eqref{E:det2}, we obtain 
	$
	\varphi_{\theta}(\det(\mathcal{F})(\zz,t))=t\det(\mathcal{F}(\zz,t))
	$
	where $\varphi$ is the rank one Drinfeld module given by $\varphi_{\theta}=\theta+(-1)^{r-1}\Delta(\bold{z})\tau$. It is  isomorphic to the Carlitz module because if  $\nu:=\sqrt[q-1]{(-1)^{r-1}\Delta_r({\bold{z}})}\in \CC_{\infty}^{\times}$ then we have $\nu\varphi_{\theta}=C_{\theta}\nu$ in $\CC_{\infty}[\tau]$.
	Moreover by \cite[Prop. 6.2(b)]{EGP14}, there exists an element $\xi\in \mathbb{F}_q(t)$ and a generator $\lambda$ of $\Ker(\exp_{\varphi})$ such that $\det(\mathcal{F}(\zz,t))=\xi s_{\varphi}(\lambda;t)$ where $s_{\varphi}(\lambda;t)\in \TT$ is the Anderson generating function of $\varphi$. Indeed, using \cite[Prop. 6.2(b)]{EGP14} and \cite[Prop. 2.5.3]{GP19}, one sees that $s_{\varphi}(\lambda;t)\in \TT^{\times}$. On the other hand, we know by \cite[Prop. 6.2.4]{GP19} that $\det(\mathcal{F}(\zz,t))\in \mathbb{T}^{\times}$ and therefore $\xi \in \mathbb{F}_q^{\times}$. Thus, without loss of generality, we can assume that 
	\begin{equation}\label{E:det3}
	\det(\mathcal{F}(\zz,t))=s_{\varphi}(\lambda;t)
	\end{equation}
	where $\lambda$ is a generator of $\Ker(\exp_{\varphi})$. Taking the residue at $t=\theta$ of both sides of $\eqref{E:det3}$, we obtain by Proposition \ref{P:Anderson}(ii) that
	\begin{equation}\label{E:det4}
	\Res_{t=\theta}\det(\mathcal{F}(\zz,t))=-\lambda=-\frac{\tilde{\pi}}{u}=-\frac{\tilde{\pi}}{\sqrt[q-1]{(-1)^{r-1}\Delta_r({\bold{z}})}}.
	\end{equation}
	Hence combining \eqref{E:det6} with \eqref{E:det4} imply that 
	\begin{equation}\label{E:det5}
	\Res_{t=\theta}\det(\mathcal{F}(\zz,t))=-\tilde{\pi}^{-(q+\dots+q^{r-1})}h_r(\zz)^{-1}.
	\end{equation}
	Consider the function $F(\zz,t):=\det(\mathcal{F}(\zz,t))h_r(\zz)\tilde{\pi}^{1+q+\dots+q^{r-1}}$ for $\zz\in \Omega^r$ and $t\in \mathbb{D}_{\inorm{\theta}}$. For any $\zz\in \Omega^r$, we have by Proposition \ref{P:Anderson}(iii) that $F(\zz,t)\in \TT$. Using \eqref{E:det6} and \eqref{E:det2}, it is easy to see that $F^{(1)}(\zz,t)=(t-\theta)F(\zz,t)$. By \cite[Prop. 6.2(b)]{EGP14}, there exists an element $c(\zz,t)\in \mathbb{F}_q(t)$ such that $F(\zz,t)=c(\zz,t)\omega(t)$. Note that for a fixed element $t\in \mathbb{D}_{\inorm{\theta}}$ which is transcendental over $\mathbb{F}_q$, $F(\zz,t)$ is a rigid analytic function on $\Omega^r$ whose values are non-zero and in a discrete set, namely in $\{c\omega(t)\ \ |c\in \mathbb{F}_q(t)\}$. Thus, the connectedness of $\Omega^r$ implies that $c(\zz,t)$ is a constant and does not depend on $\zz$. We simply set $c(t):=c(\zz,t)$ and show its value in the rest of the proof. Let us fix $\zz\in \Omega^r$. Then by Proposition \ref{P:Anderson}(ii) and the above discussion, we have 
	\begin{align*}
	\Res_{t=\theta}\det(\mathcal{F}(\zz,t))&=\lim_{t\to \theta}(t-\theta)\det(\mathcal{F}(\zz,t))\\
	&=\tilde{\pi}^{-(1+q+\dots+q^{r-1})}h_r(\zz)^{-1}\lim_{t\to \theta}(t-\theta)c(t)\omega(t)\\
	&=-c(\theta)\tilde{\pi}h_r(\zz)^{-1}\tilde{\pi}^{-(1+q+\dots+q^{r-1})}\\
	&=-c(\theta)h_r(\zz)^{-1}\tilde{\pi}^{-(q+\dots+q^{r-1})}.
	\end{align*}
	Now combining with \eqref{E:det5}, we obtain $c(\theta)=1$ and hence the result.
\end{proof}

\subsection{The setting of Pellarin} In this subsection, we briefly explain the ideas of Pellarin \cite[Sec. 8 and 9]{Pel19} and then apply them to our case. Fix $\zz=(z_1,\dots,z_r)^{\intercal}\in \Omega^r$ . We define the column matrix $\omega_{\zz A}\in \Mat_{r\times 1}(\TT)$ given by 
\[
\omega_{\zz A}:=\begin{pmatrix}
s_1(\zz;t)\\
\vdots\\
s_r(\zz;t)
\end{pmatrix}.
\] 
We define \textit{the Perkins' series} $\Psi_{\zz A}(Z)$ introduced by Perkins \cite{Per13} and generalized by Pellarin \cite[Sec. 8]{Pel19}  to 
\[
\Psi_{\zz A}(Z):=\sum_{a_1,\dots,a_r\in A}{\vphantom{\sum}}' \frac{1}{Z-a_1z_1-\dots-a_rz_r}(a_1(t),\dots,a_r(t)), \ \ Z\in \CC_{\infty}.
\]
Since $\zz A$ is strongly discrete, there exists $c\in \CC_{\infty}^{\times}$ such that  $\mathbb{D}_{\inorm{c}}\cap \zz A=\{0\}$. Therefore it can be shown \cite[Lem. 8.3]{Pel19} that for any $Z\in \mathbb{D}_{\inorm{c}}$, we have 
\begin{equation}\label{E:set1}
\Psi_{\zz A}(Z)=-\sum_{\substack{j\geq 1\\j\equiv 1 \mod {q-1} }}Z^{j-1}\mathcal{E}_{\zz}(j,t)^{\intercal}.
\end{equation}
 For any element $f=\sum_{i\geq 0}a_iZ^i\in \TT[[Z]]$ and $j\in \ZZ_{\geq 0}$, we define $\boldsymbol{\tau}^{j}(f):=\sum_{i\geq 0}a_i^{(j)}Z^{iq^j}$. Recall the exponential series of $\phi^{\zz}$ given as in \eqref{E:expofz}. By the abuse of notation, we define the $\TT[[Z]]$-valued function $\exp_{\phi^{\zz}}:\TT[[Z]]\to \TT[[Z]]$ by $\exp_{\phi^{\zz}}(f):= \sum_{i\geq 0}\alpha_i(\zz)\boldsymbol{\tau}^{i}(f)$. We further set 
\[
\mathcal{H}(Z):=\Psi_{\zz A}(Z)\omega_{\zz A}=\exp_{\phi^{\zz}}(Z)^{-1}\exp_{\phi^{\zz}}(Z/(\theta-t))\in \TT[[Z]]
\]
where the last equality follows from \cite[Lem. 8.6]{Pel19}. Considering $Z$ as a variable over $\TT$, for any $n\in \ZZ_{\geq 0}$, we define the $\TT$-linear higher derivative operator $\mathcal{D}_{n}:\TT[[Z]]\to \TT[[Z]]$ by  
\[
\mathcal{D}_n(Z^m):=\begin{cases} \binom{m}{n}Z^{m-n} & \text{ if } m\geq n\\
0 & \text{otherwise}
\end{cases}.
\]
Now for $k\in \ZZ_{\geq 0}$, set
$
\mathcal{H}_k(Z):=(\mathcal{D}_{q^k-1}(\mathcal{H}(Z)))^{(-k)}.
$
Observe that $\mathcal{H}_0(Z)=\mathcal{H}(Z)$. Next, we introduce 
\[
\Omega_{\zz A}:=(\omega_{\zz A},\omega_{\zz A}^{(-1)},\dots,\omega_{\zz A}^{(1-r)})\in \GL_{r}(\TT)
\]
where the invertibility of $\Omega_{\zz A}$ follows from \cite[Prop. 6.2.4]{GP19}. We also set
\[
\mathcal{H}_{\zz A}(Z):=(\mathcal{H}_0(Z),\dots,\mathcal{H}_{r-1}(Z))\in \Mat_{1\times r}(\TT[[Z]]).
\]
Combining \cite[Thm. 8.8]{Pel19} with \eqref{E:set1}, we obtain the next theorem.
\begin{theorem}\cite[Lem. 8.3 and Thm. 8.8]{Pel19}\label{T:Pel}
	We have the following identity:
	\[
	\sum_{\substack{j\geq 1\\j\equiv 1 \mod {(q-1)} }}Z^{j-1}\mathcal{E}_{\zz}(j,t)^{\intercal}=-\mathcal{H}_{\zz A}(Z)\Omega^{-1}_{\zz A}.
	\]
\end{theorem}
For the rest of this subsection, we apply Theorem \ref{T:Pel} to deduce, under a certain condition, the relation in Theorem \ref{T:main} between the coordinates of $\mathcal{E}_{\zz}(1,t)$ and Anderson generating functions.  Since $\zz$ is fixed, to ease the notation, let us set $\mathcal{F}:=\mathcal{F}(\zz,t)$. Define the following anti-diagonal matrix:
\[
V:=\begin{pmatrix}
0&\dots & \dots& 0 &1\\
0&\dots  &0 & 1 & 0\\
\vdots & &\reflectbox{$\ddots$}  & & \vdots\\
0 &1 & 0&\dots  &0\\
1&0&\dots &\dots &0
\end{pmatrix}\in \Mat_r(\mathbb{F}_q).
\]
Observe that we have
\begin{equation}\label{E:set3}
\Omega_{\zz A}=\mathcal{F}^{(-(r-1))}V.
\end{equation}
On the other hand, taking transpose of both sides of \eqref{E:det2} and applying the automorphism $\tau^{-(r-1)}$ imply that 
\[
\mathcal{F}^{(-(r-1))}=\mathcal{F}({\Phi_{\zz}^{\intercal}}^{(-1)})^{-1}\dots ({\Phi_{\zz}^{\intercal}}^{(-(r-1))})^{-1}.
\]
Thus \eqref{E:set3} yields
$
\Omega_{\zz A}^{-1}=V(\Phi_{\zz}^{\intercal})^{(-r+1)}\dots (\Phi_{\zz}^{\intercal})^{(-1)}\mathcal{F}^{-1}.
$
After taking the transpose and applying Theorem \ref{T:Pel}, we now get
\begin{equation}\label{E:set4}
\sum_{\substack{j\geq 1\\j\equiv 1 \mod {q-1} }}Z^{j-1}\mathcal{E}_{\zz}(j,t)=-(\mathcal{F}^{-1})^{\intercal}\Phi_{\zz}^{(-1)}\dots \Phi_{\zz}^{(-r+1)}V\begin{pmatrix}
\mathcal{H}(Z)\\
\mathcal{H}_1(Z)\\
\vdots\\
\mathcal{H}_{r-1}(Z)
\end{pmatrix}.
\end{equation}

Note that comparing the coefficients of $Z^{j-1}$ on both sides of \eqref{E:set4} introduces a relation between $\mathcal{E}_{\zz}(j,t)$ and Anderson generating functions. Concerning our purposes, we focus on the case $j=1$. For $0\leq i \leq r-1$, we let $D_{i}(\zz,t)\in\TT$ be the $Z^{0}$-th coefficient of the polynomial $\mathcal{H}_i(Z)^{(i)}=\mathcal{D}_{q^i-1}(\mathcal{H}(Z))$ which is indeed, by the definition of higher derivatives, the coefficient of $Z^{q^i-1}$ in the $Z$-expansion of $\mathcal{H}(Z)$. For $r\geq 2$ and $1\leq j \leq r-1$, we set
\[
G_{j,r}(\zz,t):=\frac{(t-\theta)D_{j}(\zz,t)-g_1(\zz)D_{j-1}^{(1)}(\zz,t)-\dots-g_j(\zz)D_0^{(j)}(\zz,t)}{\Delta_r(\zz)}\in \TT\cap \CC_{\infty}(t).
\]
Note that we consider $G_{j,r}(\cdot,t):\Omega^r\to\TT$ as a $\TT$-valued function on $\Omega^r$. By the abuse of notation, for a fixed choice of $t\in \CC_{\infty}$ such that $\inorm{t}\leq 1$, we also consider $G_{j,r}(\cdot,t):\Omega^r\to \CC_{\infty}$ to be a $\mathbb{C}_\infty$-valued function on $\Omega^r$.

Note that \eqref{E:set4} implies 
\begin{equation}\label{E:proof2}
\begin{split}
\mathcal{E}_{\zz}(1,t)&=-(\mathcal{F}^{-1})^{\intercal}\Phi_{\zz}^{(-1)}\dots \Phi_{\zz}^{(-(r-1))}\begin{pmatrix} D_{r-1}(\zz,t)\\ D_{r-2}(\zz,t)^{(1)}\\ \vdots \\D_0(\zz,t)^{(r-1)}
\end{pmatrix}^{(-(r-1))}\\
&=-(\mathcal{F}^{-1})^{\intercal}\Phi_{\zz}^{(-1)}\dots \Phi_{\zz}^{(-(r-2))}\begin{pmatrix} D_{r-2}(\zz,t)\\ D_{r-3}(\zz,t)^{(1)}\\ \vdots \\G_{r-1,r}(\zz,t)^{(-1)}
\end{pmatrix}^{(-(r-2))}\\
&\ \ \vdots\\
&=-(\mathcal{F}^{-1})^{\intercal}\begin{pmatrix} D_{0}(\zz,t)^{(1)}\\ H_0(\zz,t)\\ \vdots\\ H_{r-2}(\zz,t)
\end{pmatrix}^{(-1)}
\end{split}
\end{equation}
where we set $H_0(\zz,t):=G^{(-(r-2))}_{r-1,r}(\zz,t)$ and for $1\leq k \leq r-2$ we define
\[
H_k(\zz,t):=G^{(-(r-(k+2)))}_{r-(k+1),r}(\zz,t)-\sum_{l=1}^{k}\bigg(\frac{g_{r-l}(\zz)}{\Delta_r(\zz)}\bigg)^{(-(r-(k+2)))}H_{k-l}(\zz,t).
\]

\subsection{An analysis on the polynomial $\mathcal{H}(Z)$}
In this subsection, we analyze the coefficients of $\mathcal{H}(Z)=\exp_{\phi^{\zz}}(Z)^{-1}\exp_{\phi^{\zz}}(Z/(\theta-t))$. We set $\mathcal{H}(Z)=\sum_{i\geq 0}c_i(\zz,t)Z^i$ together with the convention $c_i(\zz,t)=0$ for $i<0$ and recall that $\exp_{\phi^{\zz}}=\sum_{i\geq 0}\alpha_i(\zz)\tau^i$. We have the following recurrence relation among coefficients $c_i(\zz,t)\in \TT$ which can be seen by the direct calculation: 
\begin{equation}\label{E:pol1}
c_i(\zz,t)=\begin{cases}\frac{1}{\theta-t} & \text{ if } i=0,\\
-\sum_{j=1}^{l-1}\alpha_j(\zz)c_{i-(q^j-1)}(\zz,t) & \text{ if } q^{l-1}-1<i<q^l-1, \ \  l\in \ZZ_{\geq 1},\\
\frac{\alpha_l(\zz)}{\theta^{q^l}-t}-\sum_{j=1}^{l}\alpha_j(\zz)c_{i-(q^j-1)}(\zz,t) & \text{ if } i=q^l-1, \ \ l \in \ZZ_{\geq 1}.
\end{cases}
\end{equation}

Using \eqref{E:pol1}, we prove the following.
\begin{lemma}\label{L:1} Let $i$ be a non-negative integer.
	\begin{itemize} 
		\item[(i)]  If $i$ is not divisible by $q-1$, then $c_i(\zz,t)=0$.
		\item[(ii)] If $q-1$ divides $i$, then $c_i(\zz,t)$ may be written as a $K(t)$-linear combination of Drinfeld modular forms of weight $i$ and type 0.
	\end{itemize}
\end{lemma}
\begin{proof} To prove the first part, let $i=j+k(q-1)$ for some $0<j<q-1$ and $k\in \ZZ_{\geq 0}$. We do induction on $k$. If $k=0$, then \eqref{E:pol1} implies that $c_i=0$. Assume that the statement in the first part of the lemma holds for $k\in \ZZ_{\geq 0}$. Now consider $i=j+(k+1)(q-1)$. By \eqref{E:pol1}, we have that 
	\[
	c_i(\zz,t)=-\alpha_1(\zz)c_{j+k(q-1)}(\zz,t)-\alpha_2(\zz)c_{j+(k-q)(q-1)}(\zz,t)-\dots-\alpha_{l-1}(\zz)c_{k+(k-q-\dots-q^{l-2})(q-1)}(\zz,t)
	\]
	where $l\in \ZZ_{\geq 1}$ such that $q^{l-1}<i<q^{l}-1$. By the induction hypothesis, we deduce that $c_i(\zz,t)=0$. To prove the second part, we set $i=k(q-1)$ for some $k\in \ZZ_{\geq 0}$. We again do induction on $k$. If $k=0$, then $c_0(\zz,t)=\frac{1}{\theta-t}$ and we are done. Assume that the statement in the second part holds for $k\in \ZZ_{\geq 0}$. Let $i=(k+1)(q-1)$.
	
	\textit{Case 1:} Assume $i=q^l-1$ for some positive integer $l$. Then by \eqref{E:pol1}, we have 
	\[
	c_i(\zz,t)=\frac{\alpha_l(\zz)}{\theta^{q^l}-t}-\alpha_1(\zz)c_{(q^l-1)-(q-1)}(\zz,t)-\dots-\alpha_l(\zz)c_0(\zz,t).
	\] 
By Example \ref{Ex:1}(ii) and the induction hypothesis, each term above has weight
$q^l-1=i$
	which concludes the proof of the first case.
	
	\textit{Case 2:} Assume that there exists $l\in \ZZ_{\geq 1}$ such that $q^{l-1}-1<i<q^l-1$. The proof of Case 2 is similar to Case 1 which follows from applying the induction on the values of $k$ and observing the following calculation
	\[
	q^m-1+(k+1)(q-1)-(q^m-1)=(k+1)(q-1)=i
	\]
	for each $1\leq m \leq l$. We leave the details to the reader.
\end{proof}
A simple consequence of Lemma \ref{L:1} can be stated as follows.
\begin{corollary}\label{C:mod} Let $r\geq 2$ be an integer and $1\leq j \leq r-1$. For any $t\in \CC_{\infty}$ with $\inorm{t}\leq 1$,      $G_{j,r}(\cdot,t):\Omega^r\to \CC_{\infty}$ is a weak modular form of weight $q^j-q^{r}$ and type 0.
\end{corollary}
\begin{proof}
	For any $0\leq k \leq j$, note that $D_{k}(\zz,t)$ is the  $Z^{q^k-1}$-st coefficient of $\mathcal{H}(Z)$. Therefore, by Lemma \ref{L:1}, it is a Drinfeld modular form of weight $q^k-1$ and type 0. The corollary now follows from Example \ref{Ex:1}(i) and \eqref{E:formula}. 
\end{proof}

    By \eqref{E:formula} and the formulas in \eqref{E:pol1}, one can see that the $\TT$-valued function $\Delta_{r}(\cdot)G_{j,r}(\cdot,t)$ can be written as a polynomial of rigid analytic functions $\alpha_1,\dots,\alpha_j:\Omega^r\to \CC_{\infty}$ with coefficients in $K(t)\cap\TT$. Hence one can observe that such polynomial only depends on $j$ and is independent of $r$.
    To be more precise, let us consider $r>r'>j\geq 1$. Let $\zz'\in \Omega^{r'}$ and for each $1\leq i \leq r'-1$, $\alpha_i':\Omega^{r'}\to \CC_{\infty}$ be the Drinfeld modular form defined as in Example \ref{Ex:1}(i) considering the Drinfeld module $\phi^{\zz'}$ corresponding to the $A$-lattice $\zz'A$. We also let $\Delta_{r'}:\Omega^{r'}\to \CC_{\infty}$ be the discriminant function defined as in Example \ref{Ex:1}(ii). Then by \eqref{E:formula} and the formulas in \eqref{E:pol1}, there is a polynomial $f_j\in (K(t)\cap\TT)[X_1,\dots,X_j]$ such that
    \[
        \Delta_{r}(\cdot)G_{j,r}(\cdot,t)=f_j(\alpha_1,\dots,\alpha_j)
    \]
    and 
    \[
        \Delta_{r'}(\cdot)G_{j,r'}(\cdot,t)=f_j(\alpha_1',\dots,\alpha_j')
    \]
    holds simultaneously as $\TT$-valued functions on $\Omega^r$ and $\Omega^{r'}$ respectively. Indeed, this result comes from the fact that the identity \eqref{E:formula} is independent of $r$ and $r'$ whenever $r>r'>j\geq 1$. 
    
    In what follows, we demonstrate how to reduce the task of proving Theorem~\ref{T:main} to the lower rank case.

\begin{proposition}\label{P:mod} 
    Let $r\geq 2$. If $G_{j,j+1}(\cdot,t)$ vanishes identically as a $\TT$-valued function on $\Omega^{j+1}$ for any $1\leq j\leq r-1$, then  $G_{j,r}(\cdot,t)$ also vanishes identically as a $\TT$-valued function on $\Omega^{r}$ for any $1\leq j\leq r-1$. In particular, we have
    \[
		\mathcal{E}_{\zz}(1,t)^{\intercal}=-\Big(
		\frac{1}{\theta-t}, 0,\dots,0\Big)\mathcal{F}^{-1}.
	\]
\end{proposition}
\begin{proof} 
    Considering the $\TT$-valued function $\Delta_{j+1}(\cdot)G_{j,j+1}(\cdot,t)$ on $\Omega^{j+1}$ and using the above discussion, we can let $f_j\in (K(t)\cap\TT)[X_1,\dots,X_j]$ be such that
    \begin{equation}\label{Eq:Relations_among_alpha_i}
        \Delta_{j+1}(\cdot)G_{j,j+1}(\cdot,t)=f_j(\alpha_1',\dots,\alpha_j')
    \end{equation}
    where $\alpha_i':\Omega^{j+1}\to \CC_{\infty}$ is the Drinfeld modular form defined as in Example \ref{Ex:1}(i) for each $1\leq i\leq j+1$ and $\Delta_{j+1}:\Omega^{j+1}\to \CC_{\infty}$ is the discriminant function defined as in Example \ref{Ex:1}(ii). Let $t\in\mathbb{C}_\infty$ with $\inorm{t}\leq 1$. Since $\Delta_{j+1}(\cdot)$ is non-vanishing on $\Omega^{j+1}$, by evaluating \eqref{Eq:Relations_among_alpha_i} at $t$ we get an algebraic relation among $\mathbb{C}_\infty$-valued functions $\alpha_1',\dots,\alpha_j'$ over $\mathbb{C}_\infty$. By \cite[Thm. 17.5(a)]{BBP18}, we know that $\alpha_1',\dots,\alpha_{r'}'$ are algebraically independent functions over $\CC_{\infty}$. Hence the coefficient of each monomial of $f_j$ must vanish. Since $t\in\mathbb{C}_\infty$ with $\inorm{t}\leq 1$ is arbitrary and such coefficients are rational functions, we conclude that  $f_j$ must be the zero polynomial.

    For the latter statement, we consider \eqref{E:proof2}. By the first part of the proposition we obtain that $G_{j,r}(\cdot,t)$ is an identically zero function for each $1\leq j \leq r-1$ and hence $H_k(\zz,t)=0$ for all $\zz\in \Omega^r$ and $0\leq k \leq r-2$. Thus, the second part of the proposition follows from the calculation in \eqref{E:proof2} and noting that $D_0(\zz,t)=\frac{1}{\theta-t}$.
\end{proof}


\subsection{Limiting behavior of the $\TT$-valued Eisenstein series}
Let $\zz=(z_1,\dots,z_r)^{\intercal}\in \Omega^r$ and $\tilde{\zz}:=(z_2,\dots,z_r)^{\intercal}\in \Omega^{r-1}$. In this subsection, we analyze the limiting behavior of the entries of $\mathcal{E}_{\zz}(1,t)$ when $\inorm{z_1}=\inorm{\zz}_i\to \infty$ where we recall the function  $\inorm{\cdot}_i$ from \S2.3.

Before stating our next lemma, for any $a\in A$ and $\zz \in \Omega^r$, we let 
\[
u_{a}(\zz):=u(a\zz)=\tilde{g}(\tilde{\zz})u^{q^{{(r-1)\deg(a_1)}}}+\mathcal{O}(u^{q^{(r-1)\deg(a_1)}})
\]
where $\tilde{g}:\Omega^{r-1}\to \CC_{\infty}$ is a certain weak modular form and the notation $\mathcal{O}(u^{q^{r-1}\deg(a)})$ represents the sum of higher degree terms in $u$ with coefficients as weak modular forms. We refer the reader to \cite[Sec. 3.5]{B14}, for further details on the $u$-expansion of $u_{a}(\zz)$.

\begin{lemma} \label{L:u}  For any $t\in \CC_{\infty}$ with $\inorm{t}\leq 1$, we have 
	\[
		\sum_{a_1,\dots,a_r\in A}{\vphantom{\sum}}'\frac{a_1(t)}{a_1z_1+\dots+a_rz_r}=\tilde{\pi}\sum_{a_1\neq 0}a_1(t)u_{a_1}(\zz).
	\]
\end{lemma}
\begin{proof}  For some $z'\in \CC_{\infty}$, taking the logarithmic derivative of both sides of $\exp_{\tilde{\pi}\tilde{\zz}}(z')=z'\prod_{\lambda \in \tilde{\pi}\tilde{\zz}A}'\Big(1-\frac{z'}{\lambda}\Big)$ implies that $\sum_{\lambda \in \tilde{\pi}\tilde{\zz}A}\frac{1}{z'-\lambda}=\exp_{\tilde{\pi}\tilde{\zz}A}(z')^{-1}$.  Now replacing $z'$ with $\tilde{\pi}a_1z_1$, we obtain
	\[
\sum_{a_1,\dots,a_r\in A}{\vphantom{\sum}}'\frac{a_1(t)}{a_1z_1+\dots +a_rz_r}=\tilde{\pi}\sum_{a_1\neq 0}a_1(t)\sum_{a_2,\dots,a_r\in A}{\vphantom{\sum}}'\frac{1}{\tilde{\pi}a_1z_1+\dots+\tilde{\pi}a_rz_r}	=\tilde{\pi}\sum_{a_1\neq 0}a_1(t)u_{a_1}(\zz).
	\]
	
\end{proof}
\begin{proposition}\label{P:Eins} Let us fix $\tilde{\zz}\in \Omega^{r-1}$. For any $2\leq j \leq r$, a non-negative integer $k$ and $t\in \CC_{\infty}$ such that $\inorm{t}\leq 1$, we have  
	\[
	\lim_{\inorm{z_1}=\inorm{\zz}_i\to \infty}\sum_{a_1,\dots,a_r\in A}{\vphantom{\sum}}'\frac{a_j(t)}{(a_1z_1+\dots +a_rz_r)^{q^k}}=\sum_{a_2,\dots,a_r\in A}{\vphantom{\sum}}'\frac{a_j(t)}{(a_2z_2+\dots +a_rz_r)^{q^k}}.
	\]
\end{proposition}
\begin{proof} The idea of the proof is due to Pellarin \cite[Lem. 25]{Pel12}. We claim that 
	\[
	\lim_{\inorm{z_1}=\inorm{\zz}_i\to \infty}\sum_{a_1\neq 0}\sum_{a_2,\dots,a_r\in A}\frac{a_j(t)}{(a_1z_1+\dots +a_rz_r)^{q^k}}=0.
	\]
	Consider an element $\zz\in \Omega^r$ such that $\inorm{z_1}=\inorm{\zz}_i$. For any $a_1\in A\setminus \{0\}$, we have $\inorm{a_1z_1}=\inorm{a_1\zz}_i\leq \inorm{a_1z_1+\dots+a_rz_r}$ where the  inequality follows from the definition of the function $\inorm{\cdot}_{i}$.
	Thus we have  
	\[
	\bigg| \frac{a_j(t)}{(a_1z_1+\dots +a_rz_r)^{q^k}}\bigg|\leq \inorm{z_1}^{-q^k}
	\]
	which proves the claim and hence the proposition.
\end{proof}

\begin{remark}\label{R:unit} We continue with an analysis, due to Gekeler \cite[(2.1)]{Gek84}, of $\exp_{\zz A}$ and its values. 
	 Let $b_1,\dots,b_r$ be polynomials in $A$ of degree less than the degree of a fixed polynomial $n$. For each $1\leq j \leq r$, set $\mu_j=b_j/n$. We have
\begin{equation}\label{E:Gek}
\begin{split}
&\exp_{\zz A}(\mu_1z_1+\dots +\mu_rz_r)\\
&=\mu_1z_1+\dots +\mu_rz_r\prod_{(a_2,\dots,a_r)\in A^{r-1}\setminus\{0\}}\Big(1-\frac{\mu_1z_1+\dots +\mu_rz_r}{a_2z_2+\dots+a_rz_r}\Big)\times \\
&\ \ \ \  \ \ \ \ \ \  \ \ \ \ \ \  \ \ \ \ \ \ \ \ \  \ \ \ \ \ \ \ \ \  \ \ \ \ \ \ \ \ \  \ \ \ \ \prod_{a_1\neq 0}\prod_{a_2,\dots,a_r\in A}\Big(1-\frac{\mu_1z_1+\dots +\mu_rz_r}{a_1z_1+a_2z_2+\dots+a_rz_r}\Big)\\
&=\exp_{\tilde{\zz}A}(\mu_1z_1+\dots +\mu_rz_r)\prod_{a_1\neq 0}\frac{\exp_{\tilde{\zz}A}(-a_1z_1+\mu_1z_1+\dots+\mu_rz_r)}{\exp_{\tilde{\zz}A}(-a_1z_1)}.
\end{split}
\end{equation}
 For further use, we need to analyze the ratio
	\[
	\frac{\exp_{\tilde{\zz}A}(-a_1z_1+\mu_1z_1+\dots+\mu_rz_r)}{\exp_{\tilde{\zz}A}(-a_1z_1)}
	\]
	for any $a_1\neq 0$. Let $\zz=(z_1,\dots,z_r)^{\intercal}\in \Omega^r$ be such that $\inorm{z_1}=\inorm{\zz}_i$. First we note that if $\inorm{z_1}=\inorm{\lambda}$ for some $\lambda\in \tilde{\zz}A$, then $\inorm{z_1-\lambda}\geq \inorm{\zz}_i=\inorm{z_1}$. On the other hand, by the properties of the non-archimedean norm $\inorm{\cdot}$, we have $\max\{\inorm{z_1},\inorm{\lambda}\}=\inorm{z_1}\geq \inorm{z_1-\lambda}$. Thus we have $\inorm{1-\frac{z_1}{\lambda}}=1$. By the product expansion of the exponential function $\exp_{\tilde{\zz} A}$, we obtain
	\[
	\inorm{\exp_{\tilde{\zz} A}(z_1)}=\inorm{z_1}\prod_{0<\inorm{\lambda}<\inorm{z_1}}\frac{\inorm{z_1}}{\inorm{\lambda}}.
	\]
	In other words, $\inorm{\exp_{\tilde{\zz} A}(z_1)}$ is determined by the norm of $z_1$ such that $\inorm{z_1}=\inorm{\zz}_i$. To conclude, under the assumption $\inorm{z_1}=\inorm{\zz}_i$ and $\inorm{z_1}>\max\{\inorm{z_2},\dots,\inorm{z_{r}}\}$, we have for any $a_1\in A\setminus \{0\}$ that 
	\[
	\frac{\inorm{\exp_{\tilde{\zz}A}(-a_1z_1+\mu_1z_1+\dots+\mu_rz_r)}}{\inorm{\exp_{\tilde{\zz}A}(-a_1z_1)}}=\frac{\inorm{\exp_{\tilde{\zz}A}((\mu_1-a_1)z_1)}}{\inorm{\exp_{\tilde{\zz}A}(-a_1z_1)}}=\frac{\inorm{\exp_{\tilde{\zz}A}(-a_1z_1)}}{\inorm{\exp_{\tilde{\zz}A}(-a_1z_1)}}=1
	\]
	where the first equality follows from $\inorm{-a_1z_1+\mu_1z_1+\dots+\mu_rz_r}=\inorm{(\mu_1-a_1)z_1}$ and the second equality is obtained by the fact that $\inorm{(\mu_1-a_1)z_1}=\inorm{-a_1z_1}$.
\end{remark}

\begin{proposition}\label{P:limit} Fix $\tilde{\zz}\in \Omega^{r-1}$ and let $t\in \CC_{\infty}$ be such that $\inorm{t}\leq 1$. 
	\begin{itemize}
		\item[(i)] For any $2\leq j \leq r$, we have 
		\begin{multline*}
		\lim_{\inorm{z_1}=\inorm{\zz}_i\to \infty}s_j^{(r-1)}(\zz;t)\sum_{a_1,\dots,a_r\in A}{\vphantom{\sum}}'\frac{a_j(t)}{(a_1z_1+\dots +a_rz_r)^{q^{r-2}}}\\
		=\tilde{s}_j^{(r-1)}(\Tilde{\zz};t)\sum_{a_2,\dots,a_r\in A}{\vphantom{\sum}}'\frac{a_j(t)}{(a_2z_2+\dots +a_rz_r)^{q^{r-2}}}
\end{multline*}	
where $\tilde{s}_j(\Tilde{\zz};t)$ is the Anderson generating function of the Drinfeld module corresponding to the lattice $\tilde{\zz}A$.
\item[(ii)] $\lim_{\inorm{z_1}=\inorm{\zz}_i\to \infty}s_1^{(r-1)}(\zz;t)u(\zz)^{q^{r-2}}=\begin{cases}
1 & \text{ if } r=2\\
0 & \text{ if } r> 2
\end{cases}.$
	\end{itemize}
\end{proposition}
\begin{proof}  By Remark \ref{R:unit} for any $l\geq 0$, we have 
	\[
	\lim_{\inorm{z_1}=\inorm{\zz}_i\to \infty}\exp_{\zz A}\Big(\frac{z_j}{\theta^{l+1}}\Big)^{q^{r-1}}=\lim_{\inorm{z_1}=\inorm{\zz}_i\to \infty}\exp_{\tilde{\zz} A}\Big(\frac{z_j}{\theta^{l+1}}\Big)^{q^{r-1}}=\exp_{\tilde{\zz} A}\Big(\frac{z_j}{\theta^{l+1}}\Big)^{q^{r-1}}.
	\]
	Thus we conclude that 
	\[
	\lim_{\inorm{z_1}=\inorm{\zz}_i\to \infty}s^{(r-1)}_j(\zz;t)=\tilde{s}^{(r-1)}_j(\Tilde{\zz};t).
	\]
	Hence by Proposition \ref{P:Eins}, we obtain the first part of the proposition. We now prove the second part. For any $l\geq 0$, set
	 \[
	 f_l(z_1):=\frac{\exp_{\tilde{\zz}A}(z_1/\theta^{l+1})}{\exp_{\tilde{\zz}A}(z_1)}.
	 \]
	 We first claim that $\lim_{\inorm{z_1}=\inorm{\zz}_i\to \infty}f_l(z_1)=0$ uniformly in $l>0$. Let $\gamma \in \tilde{\zz}A$ be such that $\inorm{\gamma}$ is the minimal among the non-zero elements of $\tilde{\zz}A$. We analyze two cases. Since $\exp_{\tilde{\zz} A}(x+y)=\exp_{\tilde{\zz} A}(x)$  for any $x\in \CC_{\infty}$ and $y\in \tilde{\zz}A$, without loss of generality, we can assume that $\inorm{z_1}> \inorm{\gamma}$.
	 
	 \textit{Case 1: } If $\inorm{\gamma}<\frac{\inorm{z_1}}{\inorm{\theta}^{l+1}}$, then we have
	 \[
	 \inorm{f_l(z_1)}=\frac{\inorm{z_1}}{\inorm{\theta}^{l+1}}\prod_{\substack{\lambda\in \tilde{\zz}A\\0<\inorm{\lambda}<\frac{\inorm{z_1}}{\inorm{\theta}^{l+1}}}}\frac{\inorm{z_1}}{\inorm{\lambda}\inorm{\theta}^{l+1}}\bigg(\inorm{z_1}\prod_{\substack{\lambda\in \tilde{\zz}A\\0<\inorm{\lambda}<\inorm{z_1}}}\frac{\inorm{z_1}}{\inorm{\lambda}}\bigg)^{-1}<\prod_{\substack{\lambda\in \tilde{\zz}A\\\frac{\inorm{z_1}}{\inorm{\theta}^{l+1}}\leq \inorm{\lambda}<\inorm{z_1}}}\frac{\inorm{\lambda}}{\inorm{z_1}}.
	 \]
	 Now we pick the largest integer $N$ such that $\inorm{\gamma}\inorm{\theta}^N<\inorm{z_1}$. Since $\inorm{\theta}=q$, by the maximality of $N$, one can show that $\floor*{\log_{q}\frac{\inorm{z_1}}{\inorm{\gamma}}}\geq N \geq l+1$ and $\inorm{\gamma}\inorm{\theta}^{N}\geq \frac{\inorm{z_1}}{\inorm{\theta}^{l+1}}$ where $\floor*{\cdot}$ is the floor function. Therefore we obtain
	 \begin{equation}\label{E:norm}
	 \inorm{f_l(z_1)}<\prod_{\substack{\lambda\in \tilde{\zz}A\\\frac{\inorm{z_1}}{\inorm{\theta}^{l+1}}\leq \inorm{\lambda}<\inorm{z_1}}}\frac{\inorm{\lambda}}{\inorm{z_1}}<\prod_{\substack{\deg(a)\in A_{+}\\\deg(a)=N}}\frac{\inorm{a}\inorm{\gamma}}{\inorm{z_1}}=\bigg(\frac{\inorm{\theta}^N\inorm{\gamma}}{\inorm{z_1}}\bigg)^{q^N}.
	 \end{equation}
	 Note that since $N$ approaches infinity as $\inorm{z_1}\to \infty$, \eqref{E:norm} implies that $\inorm{f_l(z_1)}$ approaches zero independently of $l$.
	 
	 \textit{Case 2: } If $\inorm{\gamma}\geq \frac{\inorm{z_1}}{\inorm{\theta}^{l+1}}$, then we have 
	 \[
	  \inorm{f_l(z_1)}=\frac{\inorm{z_1}}{\inorm{\theta}^{l+1}}\bigg(\inorm{z_1}\prod_{\substack{\lambda \in \tilde{\zz}A\\0<\inorm{\lambda}<\inorm{z_1}}}\frac{\inorm{z_1}}{\inorm{\lambda}}\bigg)^{-1}<\prod_{\substack{\lambda \in \tilde{\zz}A\\0<\inorm{\lambda}<\inorm{z_1}}}\frac{\inorm{\lambda}}{\inorm{z_1}}<\bigg(\frac{\inorm{\theta}^N\inorm{\gamma}}{\inorm{z_1}}\bigg)^{q^N}
	 \] 
	 where we know that the product over $\lambda\in \tilde{\zz}A$ with $0<\inorm{\lambda}<\inorm{z_1}$ is non-empty by the assumption that $\inorm{z_1}>\inorm{\gamma}$. Thus, similarly to Case 1, we see that $\inorm{f_l(z_1)}\to 0$ as $\inorm{z_1} \to  \infty$ independently of $l$.
	 
	 Since the limit is uniform, we have that 
	 \begin{equation}\label{E:lim1}
	 \begin{split}
	 \lim_{\inorm{z_1}=\inorm{\zz}_i\to \infty}u(\zz)^{q^{r-2}}\sum_{l\geq 0}\exp_{\zz A}\Big(\frac{z_1}{\theta^{l+1}}\Big)^{q^{r-1}}t^l&=\sum_{l\geq 0}\lim_{\inorm{z_1}=\inorm{\zz}_i\to \infty}u(\zz)^{q^{r-2}}\exp_{\zz A}\Big(\frac{z_1}{\theta^{l+1}}\Big)^{q^{r-1}}t^l\\
	 &=\lim_{\inorm{z_1}=\inorm{\zz}_i\to \infty}u(\zz)^{q^{r-2}}\exp_{\tilde{\zz} A}\Big(\frac{z_1}{\theta}\Big)^{q^{r-1}}.
	 \end{split}
	 \end{equation}
	 Consider the Drinfeld module $\phi^{\tilde{\zz}}$ corresponding to the $A$-lattice $\tilde{\zz}A$ given by
	 \begin{equation}\label{E:tilde}
	 \phi^{\tilde{\zz}}_{\theta}=\theta+\tilde{g}_1(\tilde{\zz})\tau+\dots +\tilde{g}_{r-2}(\tilde{\zz})\tau^{r-2}+\Delta_{r-1}(\tilde{\zz})\tau^{r-1}.
	 \end{equation}
	 Using the functional equation $\exp_{\tilde{\zz}A}\theta=\phi^{\tilde{\zz}}_{\theta}\exp_{\tilde{\zz}A}$, we obtain
	 \begin{equation}\label{E:lim2}
	 \frac{\exp_{\tilde{\zz} A}\Big(\frac{z_1}{\theta^{l+1}}\Big)^{q^{r-1}}}{\exp_{\tilde{\zz} A}(z_1)^{q^{r-2}}}=\frac{\Big(\exp_{\tilde{\zz}A}\Big(\frac{z_1}{\theta^l}\Big)-\tilde{g}_{r-2}(\tilde{\zz})\exp_{\tilde{\zz}A}\Big(\frac{z_1}{\theta^{l+1}}\Big)^{q^{r-2}}-\dots-\theta \exp_{\tilde{\zz}A}\Big(\frac{z_1}{\theta^{l+1}}\Big)\Big)}{\Delta_{r-1}(\tilde{\zz})\exp_{\tilde{\zz} A}(z_1)^{q^{r-2}}}.
	 \end{equation}
	 Thus combining \eqref{E:lim1} and \eqref{E:lim2} together with the above analysis on the limit of $f_l(z_1)$ when $z_1$ tends to infinity, we get 
	 \[
	 \lim_{\inorm{z_1}=\inorm{\zz}_i\to \infty}s_1(\zz;t)^{(r-1)}u(\zz)^{q^{r-2}}=0
	 \]
	 when $r>2$. If $r=2$ and $j=0$, we have
	 \[
	 \frac{\exp_{ A}\Big(\frac{z_1}{\theta}\Big)^{q}}{\exp_{A}(z_1)}=\frac{\exp_{A}(z_1)-\theta\exp_{A}(z_1/\theta)}{\exp_{A}(z_1)}\to 1
	 \]
	 when $\inorm{z_1}\to \infty$ which concludes the second part of the proposition.
\end{proof}
\begin{proof}[{Proof of Theorem \ref{T:main}}] Let $t\in \CC_{\infty}$ be such that  $\inorm{t}\leq 1$. We claim that the $\CC_{\infty}$-valued function $G_{r-1,r}(\zz,t)= 0$ for any $r\geq 2$ and any $\zz\in \Omega^r$. Recall that  \eqref{E:proof2} implies 
	\begin{equation}\label{E:proof3}
	\mathcal{F}^{\intercal}\mathcal{E}_{\zz}(1,t)=-\begin{pmatrix} \frac{1}{\theta-t},H_0^{(-1)}(\zz,t),\dots,H^{(-1)}_{r-2}(\zz,t)
	\end{pmatrix}^{\intercal}.
	\end{equation}
Comparing the second coordinates of both sides of \eqref{E:proof3}, we obtain
\begin{equation}\label{E:proof4}
s_1^{(1)}(\zz;t)\sum_{a_1,\dots,a_r\in A}{\vphantom{\sum}}'\frac{a_1(t)}{a_1z_1+\dots +a_rz_r}+\dots +s_r^{(1)}(\zz;t)\sum_{a_1,\dots,a_r\in A}{\vphantom{\sum}}'\frac{a_r(t)}{a_1z_1+\dots +a_rz_r}=-G_{r-1,r}^{(1-r)}(\zz,t).
\end{equation}

Let us  denote the left hand side of \eqref{E:proof4} by $\mathcal{M}(\zz,t)$. Then by Lemma \ref{L:u} and Proposition \ref{P:limit} we see that  $\lim_{\inorm{z_1}=\inorm{\zz}_i\to \infty}\mathcal{M}^{(r-2)}(\zz,t)$ is bounded. This implies by using \eqref{E:proof4} that  $\lim_{\inorm{z_1}=\inorm{\zz}_i\to \infty}G_{r-1,r}^{(-1)}(\zz,t)$ is bounded. On the other hand, by Corollary \ref{C:mod}, $G_{r-1,r}^{(-1)}(\zz,t)$ is a weak modular form of weight $q^{r-2}-q^{r-1}$. Hence $G_{r-1,r}^{(-1)}(\zz,t)$ is a Drinfeld modular form of negative weight. By \cite[Thm. 11.1]{BBP18b}, we know that there is no non-trivial Drinfeld modular form of negative weight and therefore $G_{r-1,r}^{(-1)}(\zz,t)=0$ for any $\zz\in \Omega^r$. Since it holds for any $t\in\mathbb{C}_\infty$ with $\inorm{t}\leq 1$, we deduce that $G_{r-1,r}^{(-1)}(\cdot,t)$ vanishes identically as a $\TT$-valued function on $\Omega^r$ by \cite[Thm~2.2.9]{FvdP04}. Now, as for any $j\in \ZZ$, the twisting operator sending $g\in \TT$ to $g^{(j)}\in \TT$ is an automorphism of $\TT$, we conclude that $G_{r-1,r}(\cdot,t)$ also vanishes identically as a $\TT$-valued function on $\Omega^r$. Thus by Proposition \ref{P:mod} we obtain 
\begin{equation}\label{E:proof5}
	\mathcal{E}_{\zz}(1,t)^{\intercal}=-\Big(
    \frac{1}{\theta-t}, 0,\dots,0\Big)\mathcal{F}^{-1}.
\end{equation}
The proof of the theorem now follows from Proposition \ref{P:det} and \eqref{E:proof5}.
\end{proof}

\section{The function $E_{r}$}
In this section, we introduce the function $E_r:\Omega^r\to \CC_{\infty}$ and discuss some of its properties. Let $\zz=(z_1,\dots,z_{r-1},1)^{\intercal}\in \Omega^r$  and let $\tilde{\zz}=(z_2,\dots,z_{r-1},1)^{\intercal}$. Recall the definition of $u(\zz)$ from Section 2.3. Considering $u(\zz)$ as a function of variables $z_1,\dots,z_{r-1}$, we have 
\begin{equation}\label{E:part}
\frac{\partial}{\partial z_1}u(\zz)=-\tilde{\pi}u(\zz)^2.
\end{equation}

For each $a\in A$, we define $d(a):=q^{(r-1)\deg(a)}$ if $a\in A$ with $a\neq 0$ and $d(0):=0$.
Now we consider the Drinfeld module $\phi^{\tilde{\pi}\tilde{\zz}}$ corresponding to the $A$-lattice $\tilde{\pi}\tilde{\zz}A$ and for each $a\in A$ with $a\neq 0$ we set the leading coefficient of $\phi^{\tilde{\pi}\tilde{\zz}}_a$ by $\Delta_{r-1}^{\tilde{\pi}\tilde{\zz}}(a)$.
Following \cite{BB17} with a slight modification, we define 
\[
\phi_a^{\tilde{\pi}\tilde{\zz}}(X):=aX+\dots +\Delta_{r-1}^{\tilde{\pi}\tilde{\zz}}(a)X^{d(a)}.
\]
Observe, by the definition of the Drinfeld module $\phi^{\tilde{\pi}\tilde{\zz}}$, that $\Delta_{r-1}^{\tilde{\pi}\tilde{\zz}}(a)=a$ when $a\in \FF_q$. We set $f_a(X):=X^{d(a)}\Delta_{r-1}^{\tilde{\pi}\tilde{\zz}}(a)^{-1}\phi_a^{\tilde{\pi}\tilde{\zz}}(X^{-1})$. 
Now let $\mathfrak{g}$ be a non-vanishing rigid analytic function on $\Omega^r$ and set 
\[
dL(\mathfrak{g}):=\tilde{\pi}^{-1}\mathfrak{g}^{-1}\frac{\partial}{\partial z_1}\mathfrak{g}.
\]
Using \eqref{E:part}, one can see that  $dL(u(\zz))=-u(\zz).$ Furthermore, we have 
\[
dL(f_a(u(\zz)))=\begin{cases}
a u_a(\zz) & \text{ if } a\in A\setminus\mathbb{F}_q\\
0	& \text{ otherwise}
\end{cases}.
\]
On the other hand, we have the product formula for the discriminant function given  in \cite[Thm. 5.2]{BB17} which is due to Basson:
$$\tilde{\pi}^{1-q^r}\Delta_{r}(\zz)=-\Delta_{r-1}^{\tilde{\pi}\tilde{\zz}}(\theta)^q u(\zz)^{q-1}\prod_{a\in A_+}f_a(u(\zz))^{(q^r-1)(q-1)}.$$
Note that for given non-vanishing rigid analytic functions $\mathfrak{g}_1$ and $\mathfrak{g}_2$ on $\Omega^r$, we have 
\[
dL(\mathfrak{g}_1\mathfrak{g}_2)=dL(\mathfrak{g}_1)+dL(\mathfrak{g}_2)
\]
and  for each $c\in\mathbb{C}_\infty^\times$ we have $dL(c \mathfrak{g}_1)=dL(\mathfrak{g}_1)$.
Now we define our function $E_r:\Omega^r\to \CC_{\infty}$ which is the ``logarithmic derivative'' of $\Delta_{r}$  described as
 \begin{equation}\label{E:falseEis}
 E_r(\zz):=dL(\Delta_r(\zz))=(q-1)dL(u(\zz))+(q^r-1)(q-1)dL(\prod_{a\in A_+}f_a(u(\zz)))=\sum_{a\in A_+}au_a(\zz)
 \end{equation}
where the last equality follows from the above calculation. The $u$-expansion of  $E_r$ implies that it is a rigid analytic function which is also $\tilde{\zz}A$-periodic and holomorphic at infinity.

For the rest of this section, we focus on obtaining the functional equation satisfied by $E_r$ so that when $r=2$, it turns out to be the functional equation introduced in \cite[(8.4)]{Gek88} satisfied by the false Eisenstein series of Gekeler.

Let $\gamma=(a_{ij})\in \GL_r(A)$. By the chain rule, we first have
\begin{equation}\label{E:FE1}
\begin{split}
&\frac{\partial}{\partial z_1}(\Delta_r(\gamma\cdot \zz))\\
&=\frac{\partial \Delta_r}{\partial z_1}(\gamma\cdot \zz)\frac{\partial}{\partial z_1}\Big(\frac{a_{11}z_1+\dots+a_{1(r-1)}z_{r-1}+ a_{1r}}{a_{r1}z_1+\dots+a_{r(r-1)}z_{r-1} + a_{rr}}\Big)+\\
&\ \ \ \ \ \ \ \ \ \ \  \ \ \ \ \ \ \ \ \ \ \  \frac{\partial \Delta_r}{\partial z_2}(\gamma\cdot \zz)\frac{\partial}{\partial z_1}\Big(\frac{a_{21}z_1+\dots+a_{2(r-1)}z_{r-1}+ a_{2r}}{a_{r1}z_1+\dots+a_{r(r-1)}z_{r-1} + a_{rr}}\Big)+\dots +\\
&\ \ \ \ \ \ \ \ \ \ \  \ \ \ \ \ \ \ \ \ \ \ \frac{\partial \Delta_r}{\partial z_{r-1}}(\gamma\cdot \zz)\frac{\partial}{\partial z_1}\Big(\frac{a_{(r-1)1}z_1+\dots +a_{(r-1)(r-1)}z_{r-1}+a_{(r-1)r}}{a_{r1}z_1+\dots+a_{r(r-1)}z_{r-1} + a_{rr}}\Big)\\
&=\frac{\partial \Delta_r}{\partial z_1}(\gamma\cdot \zz)\big(a_{11}j(\gamma,\zz)-a_{r1}(a_{11}z_1+\dots+a_{1(r-1)}z_{r-1}+ a_{1r})\big)j(\gamma,\zz)^{-2}+ \\
&\ \ \ \ \ \ \ \frac{\partial\Delta_r}{\partial z_2}(\gamma\cdot \zz)\big(a_{21}j(\gamma,\zz)-a_{r1}(a_{21}z_1+\dots+a_{2(r-1)}z_{r-1}+ a_{2r})\big)j(\gamma,\zz)^{-2}+\dots+\\
&\ \ \ \ \ \ \   \frac{\partial\Delta_r}{\partial z_{r-1}}(\gamma\cdot \zz)\big(a_{(r-1)1}j(\gamma,\zz)-a_{r1}(a_{(r-1)1}z_1+\dots+a_{(r-1)(r-1)}z_{r-1}+ a_{(r-1)r})\big)j(\gamma,\zz)^{-2}.
\end{split}
\end{equation}
For $2\leq j \leq r-1$, let us consider  the function $E^{[j]}_r:\Omega^r\to \CC_{\infty}$ defined by
\begin{equation}\label{E:tildefunc}
E^{[j]}_r(\zz):=\tilde{\pi}^{-1}\Delta_r(\zz)^{-1}\frac{\partial}{\partial z_j}\Delta_r(\zz).
\end{equation}
The non-vanishing of $\Delta_{r}$ on $\Omega^r$ yields that   $E^{[j]}_r$ is well-defined. From  \eqref{E:FE1},  we obtain 
\begin{equation}\label{E:FE2}
\begin{split}
&E_r(\gamma\cdot \zz)=j(\gamma,\zz)^{-2}\Big(E_r(\gamma\cdot \zz)\big(a_{11}j(\gamma,\zz)-a_{r1}(a_{11}z_1+\dots+a_{1(r-1)}z_{r-1}+ a_{1r})\big) \\
&\ \ \ \ \ \ \ \ \ \ \ \ \ \  \ +E^{[2]}_r(\gamma\cdot \zz)\big(a_{21}j(\gamma,\zz)-a_{r1}(a_{21}z_1+\dots+a_{2(r-1)}z_{r-1}+ a_{2r})\big)+\dots \\
&\ \ \ \ \ \ \ \ \ \ \ \ \ \  \ +E^{[r-1]}_r(\gamma\cdot \zz)\big(a_{(r-1)1}j(\gamma,\zz)-a_{r1}(a_{(r-1)1}z_1+\dots+a_{(r-1)(r-1)}z_{r-1}+ a_{(r-1)r})\big)\Big).
\end{split}
\end{equation}
On the other hand, since $\Delta_{r}$ is a Drinfeld modular form of weight $q^r-1$ and of type zero, we have 
\begin{equation}\label{E:FE3}
\Delta_{r}(\gamma \cdot \zz)=j(\gamma,\zz)^{q^{r}-1}\Delta_{r}(\zz).
\end{equation}
Applying  $dL$ to both sides of \eqref{E:FE3}, we obtain
\begin{equation}\label{E:FE4}
E_r(\gamma \cdot \zz)=E_r(\zz)-j(\gamma,\zz)^{-1}\tilde{\pi}^{-1}a_{r1}.
\end{equation}
Thus combining \eqref{E:FE2} with \eqref{E:FE4}, we get
\begin{equation}\label{E:FE5}
\begin{split}
E_r(\zz)&=j(\gamma,\zz)^{-2}\Big(E_r(\gamma\cdot \zz)\big(a_{11}j(\gamma,\zz)-a_{r1}(a_{11}z_1+\dots+a_{1(r-1)}z_{r-1}+ a_{1r})\big) \\
&\ \ \ \ +E^{[2]}_r(\gamma\cdot \zz)\big(a_{21}j(\gamma,\zz)-a_{r1}(a_{21}z_1+\dots+a_{2(r-1)}z_{r-1}+ a_{2r})\big)+\dots \\
&\ \ \ \ +E^{[r-1]}_r(\gamma\cdot \zz)\big(a_{(r-1)1}j(\gamma,\zz)-a_{r1}(a_{(r-1)1}z_1+\dots+a_{(r-1)(r-1)}z_{r-1}+ a_{(r-1)r})\big)\Big)\\
&\ \ \ \ + j(\gamma,\zz)^{-1}\tilde{\pi}^{-1}a_{r1}.
\end{split}
\end{equation}
Now we consider a change of variables and apply it to \eqref{E:FE5}. Let us set $\gamma':=\gamma^{-1}=\det(\gamma)^{-1}(a'_{ij})\in \GL_r(A)$. A comparison of the coefficients $a_{ij}$ and $a_{ij}'$ for each $1\leq i,j \leq r$ implies that $a_{ij}=\det(\gamma')^{-1}c'_{ji}$ where $c'_{ji}$ is the $(j,i)$-cofactor of $\gamma'$. We further set $\zz':=\gamma \cdot \zz=:(z_1',\dots,z_{r-1}',1)^{\intercal}\in \Omega^r$. By \cite[Lem. 3.1.3]{B14}, we have 
$
j(\gamma',\zz')=j(\Id,\zz)j(\gamma,\zz)^{-1}=j(\gamma,\zz)^{-1}.
$
Together with our new variables, \eqref{E:FE5} now becomes
\begin{equation}
\begin{split}
&E_{r}(\gamma'\cdot \zz')\\
&=\det(\gamma')^{-1}j(\gamma',\zz')\Big(E_r(\zz')(c'_{11}-z_1'c'_{1r})+E^{[2]}_r(\zz')(c'_{12}-z_2'c'_{1r})\\
&\ \  \ \ \ \ \ \ \ \ \  \ \ \ \ \ \ \ \ \ \ \ \  \ \ \ \ \ \ \  +\dots+E^{[r-1]}_r(\zz')(c'_{1(r-1)}-z_{r-1}'c'_{1r})+\tilde{\pi}^{-1}c'_{1r}\Big).
\end{split}
\end{equation}
We summarize our calculation above in the following theorem.
\begin{theorem}\label{T:FE} Let $\zz=(z_1,\dots,z_{r-1},1)^{\intercal}\in \Omega^r$, $E_r$ be the function defined in \eqref{E:falseEis} and $E^{[2]}_r,\dots,E^{[r-1]}_r$ be functions given in \eqref{E:tildefunc}. Then for any $\gamma\in \GL_r(A)$, we have
	\begin{align*}
	E_{r}(\gamma \cdot \zz)&=\det(\gamma)^{-1}j(\gamma,\zz)\Big(E_r(\zz)(c_{11}-z_1c_{1r})+E^{[2]}_r(\zz)(c_{12}-z_2c_{1r})\\
	&\ \ \ \ \ \ \  \ \ \ \ \ \ \ \ \ \ \ \  \ \ \ \ \ \ \  +\dots+E^{[r-1]}_r(\zz)(c_{1(r-1)}-z_{r-1}c_{1r})+\tilde{\pi}^{-1}c_{1r}\Big)
	\end{align*}
where for each $1\leq j\leq r$, $c_{1j}$ is the $(1,j)$-cofactor of $\gamma$.
\end{theorem}
An immediate consequence of Theorem \ref{T:FE} can be stated as follows.
\begin{corollary} For any $2\leq j \leq r-1$,  $E^{[j]}_r:\Omega^r\to \CC_{\infty}$ is a rigid analytic function.
\end{corollary}
\begin{proof} Let $\gamma_j\in \GL_r(A)$ be such that 
	\begin{equation}\label{E:matrix}
	\gamma_j^{-1}=\begin{pmatrix}
	1& & & & & \\
	& \ddots & & &\\
	1 & & \ddots& &\\
	& & & \ddots &\\
	0 & & & & 1
	\end{pmatrix}
	\end{equation}
	where the only non-zero terms in the first column appear in the first and the $j$-th entry. Thus, by Theorem \ref{T:FE}, we have
	\begin{equation}\label{E:first}
	E_{r}(\gamma_j \cdot \zz)=j(\gamma_j,\zz)(E_r(\zz)+E^{[j]}_r(\zz))=E_r(\zz)+E^{[j]}_r(\zz).
	\end{equation}
Since $E_r$ is a rigid analytic function, so is $E^{[j]}_r$ by \eqref{E:first}.
\end{proof}

\begin{remark} 
Consider $\gamma=\begin{pmatrix}
	a&b \\ c&d
	\end{pmatrix}\in \GL_2(A)$. Then by the notation of Theorem \ref{T:FE}, we have $c_{11}=d$ and $c_{12}=-c$. Then for $\zz=(z,1)^{\intercal}\in \Omega^2$, we have 
	\[
	E_2(\gamma\cdot \zz)=\det(\gamma)^{-1}(cz+d)(E_2(\zz)(cz+d)-c\tilde{\pi}^{-1})
	\] 
	which is already obtained by Gekeler in \cite[(8.4)]{Gek88}.
	
\end{remark}

\section{The function $\EE_r$}
As usual, we let $\zz=(z_1,\dots,z_r)^{\intercal}$ be in $\Omega^r$. Recall also from the statement of Theorem \ref{T:main} that $C_{11}$ is the (1,1)-cofactor of the matrix $\mathcal{F}(\zz,t)$ defined in \eqref{D:F}. In this section we introduce the function $\EE_{r}:\Omega^r\times \mathbb{D}_{\inorm{\theta}}\to \CC_{\infty}$ given by 
\[
\bold{E}_r(\zz,t):=-\frac{\tilde{\pi}^{q+\dots+q^{r-1}}}{\omega^{(1)}(t)}h_r(\zz)C_{11}=-\frac{\tilde{\pi}^{q+\dots+q^{r-1}}}{\omega^{(1)}(t)}h_r(\zz)\det\begin{pmatrix}
s_2^{(1)}(\zz;t)&\dots &s_2^{(r-1)}(\zz;t)\\
\vdots & & \vdots \\
s_r^{(1)}(\zz;t)&\dots &s_r^{(r-1)}(\zz;t)
\end{pmatrix} .
\]
When $r=2$, the function $\EE_2$ is deeply studied by Pellarin in \cite{Pel11} and \cite{Pel14}. He further showed that  $\EE_2$ can be understood as an infinite series of $t$ which has an infinite radius of convergence (see \cite[Sec. 3.1]{Pel14}). Although this may not hold for $\EE_r$ when $r>2$, one can still show the following by using Proposition \ref{P:Anderson} which we leave the details to the reader.
\begin{lemma}\label{L:E1}  For any fixed $\zz\in \Omega^r$, $\EE_r(\zz,\cdot)$ can be considered as a function of $t$ whose domain of convergence is  $\mathbb{D}_{\inorm{\theta^{q^2}}}$.
\end{lemma}
The next corollary is crucial to deduce important properties of the function $\EE_r$.
\begin{corollary}\label{C:EC1} For any $\zz=(z_1,\dots,z_r)^{\intercal} \in \Omega^r$ and $t\in \mathbb{D}_{\inorm{\theta^{q^2}}}$, we have 
	\begin{equation}\label{E:Euexpan}
	\EE_r(\zz,t)=\sum_{a\in A_{+}}a(t)u_{a}(\zz).
	\end{equation}
\end{corollary}
\begin{proof} By Lemma \ref{L:E1}, we see that Theorem \ref{T:main} gives an analytic continuation of the first entry of the Eisenstein series $\mathcal{E}_{\zz}(1,t)$ for  $t\in \mathbb{D}_{\inorm{\theta^{q^2}}}$. Furthermore, note that
	\[
	\EE_r(\zz,t)=-\frac{1}{\tilde{\pi}}\sum_{a_1,\dots,a_r\in A}{\vphantom{\sum}}'\frac{a_1(t)}{a_1z_1+\dots+a_rz_r}=-\sum_{a_1\neq 0}a_1(t)u_{a_1}(\zz)=\sum_{a_1\in A_{+}}a_1(t)u_{a_1}(\zz)
	\]
	where the first equality follows from Theorem \ref{T:main}, the second equality from the same calculation as in the proof of Lemma \ref{L:u} and the third equality from the identity
	$
	\sum_{c\in \mathbb{F}_q^{\times}}1=-1
	$. 
\end{proof}

 Recall that for any element $\zz \in \Omega^r$,  $\phi^{\zz}$ is the Drinfeld module corresponding to the $A$-lattice $\zz A$ and let  $F^{\phi^{\zz}}_{\delta_i}:\CC_{\infty}\to \CC_{\infty}$ be the quasi-periodic function corresponding to the biderivation $\delta_i:\theta\to \tau^i$ for each $1\leq i \leq r-1$.

\begin{theorem}\label{T:EGEK} Let $\zz=(z_1,\dots,z_r)^{\intercal}$ be an element in $\Omega^r$ and $E_r$ be the function defined as in \eqref{E:falseEis}. 
	\begin{itemize}
		\item[(i)] We have 
		$
		\EE_r(\zz,\theta)=E_r(\zz).
		$
		\item[(ii)] The following equality holds:
		\[
		E_r(\zz)=\tilde{\pi}^{-1+q+\dots+q^{r-1}}h_r(\zz)\det\begin{pmatrix}
		F^{\phi^{\zz}}_{\delta_1}(z_2)&\dots &F^{\phi^{\zz}}_{\delta_{r-1}}(z_2)\\
		\vdots & & \vdots \\
		F^{\phi^{\zz}}_{\delta_1}(z_r)&\dots &F^{\phi^{\zz}}_{\delta_{r-1}}(z_r)
		\end{pmatrix} .
	\]
	
	\end{itemize}
\end{theorem}
\begin{proof} The first part follows from Corollary \ref{C:EC1} and the definition of $E_r(\zz)$. For the second part, note that by Corollary \ref{C:EC1} and the first part, we have 
	\[
	E_r(\zz)=-\frac{\tilde{\pi}^{q+\dots+q^{r-1}}}{\omega^{(1)}(t)_{|t=\theta}}h_r(\zz)\det\begin{pmatrix}
	s_2^{(1)}(\zz;t)_{|t=\theta}&\dots &s_2^{(r-1)}(\zz;t)_{|t=\theta}\\
	\vdots & & \vdots \\
	s_r^{(1)}(\zz;t)_{|t=\theta}&\dots &s_r^{(r-1)}(\zz;t)_{|t=\theta}
	\end{pmatrix}.
	\]
	Note that $\omega^{(1)}(t)_{|t=\theta}=-\tilde{\pi}$. Now we conclude the proof by Proposition \ref{P:quasiper}.
\end{proof}

The next theorem is useful to detect an equality similar to Theorem \ref{T:EGEK}(i) for the functions $E^{[2]}_r,\dots,E^{[r-1]}_r$ introduced in \eqref{E:tildefunc}. Before we state the next theorem, for any $\gamma=(a_{ij})\in \GL_r(A)$, set $\gamma(t):=(a_{ij}(t))\in \GL_r(\mathbb{F}_q[t])$.
\begin{theorem}\label{T:FE3} Let $\zz=(z_1,\dots,z_r)^{\intercal} \in \Omega^{r}$ and $s_1(\zz;t),\dots,s_r(\zz;t)$ be the Anderson generating functions of the Drinfeld module $\phi^{\zz}$. For any $1\leq i \leq r$, let $C_{i1}$ be the $(i,1)$-cofactor of the matrix $\mathcal{F}(\zz,t)$ defined in \eqref{D:F}. Furthermore set 
	\[
	\EE^{[j]}_r(\zz,t):=-\frac{\tilde{\pi}^{q+\dots+q^{r-1}}}{\omega^{(1)}(t)}h_r(\zz)C_{j1}, \ \  j=2,\dots,r-1.
	\]
Then for any $\gamma\in \GL_r(A)$, we have

\begin{align*}
\EE_{r}(\gamma \cdot \zz,t)&=\det(\gamma)^{-1}j(\gamma,\zz)\bigg(\EE_r(\zz,t)\Big(\bar{c}_{11}-\frac{s_1(\zz;t)}{s_r(\zz;t)}\bar{c}_{1r}\Big)+\EE^{[2]}_r(\zz,t)\Big(\bar{c}_{12}-\frac{s_2(\zz;t)}{s_r(\zz;t)}\bar{c}_{1r}\Big)\\
&\ \ \ \ \  \ \ \ \ \ \ \ \ \ \ \ \ \ \ \ \ \ \ \  +\dots+\EE^{[r-1]}_r(\zz,t)\Big(\bar{c}_{1(r-1)}-\frac{s_{r-1}(\zz;t)}{s_r(\zz;t)}\bar{c}_{1r}\Big)-\frac{\bar{c}_{1r}}{\tilde{\pi}(t-\theta)s_r(\zz;t)}\bigg)
\end{align*}
where for each $1\leq j\leq r$, $\bar{c}_{1j}$ is the $(1,j)$-cofactor of $\gamma(t)$.
\end{theorem}
\begin{proof}
    The proof uses similar ideas as in \cite[Lem. 2.4]{Pel14}. We have
    \begin{align*}
        s_i(\gamma\cdot\zz;t)&=\sum_{j\geq 0}\frac{\alpha_j(\gamma\cdot\zz)({a_{i1}z_1+\cdots+a_{ir}z_r})^{q^j}}{(\theta^{q^j}-t)j(\gamma,\zz)^{q^j}}\\
        &=\sum_{j\geq 0}j(\gamma,\zz)^{q^j-1}\frac{\alpha_j(\zz)({a_{i1}z_1+\cdots+a_{ir}z_r})^{q^j}}{(\theta^{q^j}-t)j(\gamma,\zz)^{q^j}}.\\
        &=j(\gamma,\zz)^{-1}\left(a_{i1}(t)s_1(\zz;t)+\cdots+a_{ir}(t)s_r(\zz;t)\right)
    \end{align*}
    where the first equality follows from Proposition \ref{P:Anderson}(i),  the second equality from Example \ref{Ex:1}(i) and the last equality from Proposition \ref{P:Anderson}(v). Thus, we immediately derive that (cf. \cite[Lem.~2.4]{Pel14})
$
        \mathcal{F}(\gamma\cdot\zz,t)=\gamma(t)\mathcal{F}(\zz,t)\boldsymbol{j}(\gamma,\zz)
$
    where 
    \begin{align*}
        \boldsymbol{j}(\gamma,\zz)=\begin{pmatrix}
        j(\gamma,\zz)^{-1} &  &  & \\
         & j(\gamma,\zz)^{-q} &  & \\
         & & \ddots & \\
         & & & j(\gamma,\zz)^{-q^{r-1}}
        \end{pmatrix}\in\GL_r(\mathbb{C}_\infty).
    \end{align*}
    As a consequence, we get
    \begin{equation}\label{Eq:Deformation_of_Vectorial_Forms}
        (\mathcal{F}^{-1}(\gamma\cdot\zz,t))^{\intercal}=(\gamma^{-1}(t))^{\intercal}(\mathcal{F}^{-1}(\zz,t))^{\intercal}(\boldsymbol{j}^{-1}(\gamma,\zz))^{\intercal}.
    \end{equation}
    Note that by Example \ref{Ex:1}(iii) and Proposition~\ref{P:det} we have
    \begin{equation}\label{Eq:Ratio_of_Det}
        \frac{\det(\mathcal{F}(\gamma\cdot\zz,t))}{\det(\mathcal{F}(\zz,t))}=\frac{h_r(\zz)}{h_r(\gamma\cdot\zz)}=j(\gamma,\zz)^{-(1+q+\cdots+q^{r-1})}\det(\gamma).
    \end{equation}
    Let $C_{i1}^\gamma$ be the $(i,1)$-cofactor of the matrix $\mathcal{F}(\gamma\cdot\zz,t)$ for each $1\leq i\leq r$. By comparing the first column of both sides of (\ref{Eq:Deformation_of_Vectorial_Forms}), we obtain
    \begin{align*}
        \det(\mathcal{F}(\gamma\cdot\zz,t))^{-1}
        \begin{pmatrix}
        C_{11}^\gamma \\
        \vdots\\
        C_{r1}^\gamma
        \end{pmatrix}=
        (\gamma^{-1}(t))^{\intercal}
        \det(\mathcal{F}(\zz,t))^{-1}
        \begin{pmatrix}
        C_{11} \\
        \vdots\\
        C_{r1}
        \end{pmatrix}j(\gamma,\zz).
    \end{align*}
    Then (\ref{Eq:Ratio_of_Det}) implies that
    \begin{align*}
        \begin{pmatrix}
        C_{11}^\gamma \\
        \vdots\\
        C_{r1}^\gamma
        \end{pmatrix}=
        j(\gamma,\zz)^{-(q+\cdots+q^{r-1})}\det(\gamma)(\gamma^{-1}(t))^{\intercal}
        \begin{pmatrix}
        C_{11} \\
        \vdots\\
        C_{r1}
        \end{pmatrix}.
    \end{align*}
    On the other hand, we may express
$
        \det(\mathcal{F}(\zz,t))=s_1(\zz;t)C_{11}+\cdots+s_r(\zz;t)C_{r1}.
$
    Then we conclude from Proposition~\ref{P:det} that
    \begin{align*}
        C_{r1}=\frac{1}{s_r(\zz;t)}\left(\Tilde{\pi}^{-\frac{q^r-1}{q-1}}h_r(\zz)^{-1}\omega(t)-s_1(\zz;t)C_{11}-\cdots-s_{r-1}(\zz;t)C_{(r-1)1}\right).
    \end{align*}
    Finally, we get
    \begin{align*}
        \EE_r(\gamma\cdot\zz,t)&=\frac{-\Tilde{\pi}^{q+\cdots+q^{r-1}}}{\omega(t)}h_r(\gamma\cdot\zz)C_{11}^\gamma\\
        &=\frac{-\Tilde{\pi}^{q+\cdots+q^{r-1}}}{\omega(t)}j(\gamma,\zz)\det(\gamma)^{-1}h_r(\zz)\left(\bar{c}_{11}C_{11}+\cdots+\bar{c}_{1r}C_{r1}\right)\\
        &=\det(\gamma)^{-1}j(\gamma,\zz)\Big(\EE_r(\zz,t)\big(\bar{c}_{11}-\frac{s_1(\zz;t)}{s_r(\zz;t)}\bar{c}_{1r}\big)+\EE^{[2]}_r(\zz,t)\big(\bar{c}_{12}-\frac{s_2(\zz;t)}{s_r(\zz;t)}\bar{c}_{1r}\big)\\
        &\ \ \ \ \ \ \ \  \ \ \ \ \ \ \ \ \ \ \ \ \ +\dots+\EE^{[r-1]}_r(\zz,t)\big(\bar{c}_{1(r-1)}-\frac{s_{r-1}(\zz;t)}{s_r(\zz;t)}\bar{c}_{1r}\big)-\frac{\bar{c}_{1r}}{\tilde{\pi}(t-\theta)s_r(\zz;t)}\Big).
    \end{align*}
    
\end{proof}
\begin{remark} If we let $\gamma=\begin{pmatrix}
	a&b \\ c&d
	\end{pmatrix}\in \GL_2(A)$, then we have $\bar{c}_{11}=d(t)$ and $\bar{c}_{12}=-c(t)$ and therefore for $\zz=(z,1)^{\intercal}\in \Omega^2$, we have 
	\[
	\EE_2(\gamma\cdot \zz,t)=\det(\gamma)^{-1}(cz+d)\Big(\EE_2(\zz,t)\Big(c(t)\frac{s_1(\zz;t)}{s_2(\zz;t)}+d(t)\Big)+\frac{c(t)}{\tilde{\pi}(t-\theta)s_2(\zz;t)}\Big)
	\] 
	which is already obtained by Pellarin in \cite[Prop. 3.2]{Pel14}.
\end{remark}
The following corollary is a consequence of a comparison between the functional equations in Theorem \ref{T:FE} and Theorem \ref{T:FE3}.

\begin{corollary} Let $\zz$ be an element in $\Omega^r$. For any $2\leq j \leq r-1$, we have
	\[
	\EE^{[j]}_r(\zz,\theta)=E^{[j]}_r(\zz).
	\]
\end{corollary}
\begin{proof}By the definition of $\EE^{[j]}_r(\zz,t)$ and  Proposition \ref{P:Anderson}(iii), one can see that  $\EE^{[j]}_r(\zz,t)$ is well-defined at $t=\theta$. Let $\gamma_j\in \GL_r(A)$ be defined as in \eqref{E:matrix}. By Theorem \ref{T:FE3}, we have
\begin{equation}\label{E:second}
\EE_{r}(\gamma_j \cdot \zz,t)=j(\gamma_j,\zz)(\EE_r(\zz,t)+\EE^{[j]}_r(\zz,t))=\EE_r(\zz,t)+\EE^{[j]}_r(\zz,t).
\end{equation}
Since, by Theorem \ref{T:EGEK}(i), $\EE_r(\zz,\theta)=E_r(\zz)$, combining \eqref{E:first} and \eqref{E:second} implies the desired equality. 
\end{proof}

\section{Transcendence of the special values of $E_r$}
    In this section, we analyze transcendental properties of the function $E_r$ at \textit{CM points} which we will define as follows. 
    \begin{definition} We say that an element $\mathbf{z}\in\Omega^r$ is \textit{a CM point} if the Drinfeld module $\phi^{\zz}$ corresponding to the $A$-lattice $\mathbf{z}A$ is a CM Drinfeld module.
    \end{definition}
    We briefly explain a result, due to Chang and Papanikolas \cite{CP12}, about the algebraic independence of periods and quasi-periods of Drinfeld modules. Let $\phi$ be a CM Drinfeld module defined over $\overline{K}$ and $P$ be its period matrix as in \eqref{E:period}. We set $\overline{K}(P)$ to be the field generated by the entries of $P$ over $\overline{K}$. Then the transcendence degree $\trdeg_{\overline{K}}\overline{K}(P)$ of the extension $\overline{K}(P)$ over $\overline{K}$ is $r$ 
    \cite[Thm. 1.2.2]{CP12} and by \cite[(3.12), (3.13)]{BP02}, we know that the set $\mathcal{S}:=\{w_r,F^{\phi}_{\delta_{1}}(w_r),\dots,F^{\phi}_{\delta_{r-1}}(w_r)\}$ generates $\overline{K}(P)$ over $\overline{K}$. Combining these two results, we deduce that $\mathcal{S}$ forms a transcendence basis for $\overline{K}(P)$ over $\overline{K}$.
    
    Before stating our next proposition, for any $1\leq i,j \leq r$, we let $Q_{ij}$ be the $(i,j)$-cofactor of $P$.
    
    \begin{proposition}\label{Prop:Transcendence_of_Special_Values} 
	 Let $\phi$ be a CM Drinfeld module of rank $r$ defined over $\overline{K}$ given as in \eqref{intro0} whose $r$-th coefficient is $(-1)^{r-1}$ and $P$ be its period matrix. If $Q_{11}\neq 0$, then $\tilde{\pi}^{-2}w_rQ_{11}$ is transcendental over $\overline{K}$.
    \end{proposition}
    
    \begin{proof} For each $0\leq i \leq r-1$, recall the definition of $F^{\phi}_{\delta_{j}}$ from Section 2.2 and note that  $F^{\phi}_{\delta_0}(w_i)=w_i$ for each $1\leq i\leq r$. Then by \cite[(3.12), (3.13)]{BP02} we know that 
    $F^{\phi}_{\delta_{j}}(w_i)=\mathcal{L}_{ij}(w_r,F^{\phi}_{\delta_{1}}(w_r),\dots,F^{\phi}_{\delta_{r-1}}(w_r))$
    	where $\mathcal{L}_{ij}(X_0,\dots,X_{r-1})\in\overline{K}[X_0,\dots,X_{r-1}]$ is a homogeneous polynomial of total degree $1$. In particular, $\mathcal{L}_{rj}=X_j$ for $0\leq j\leq r-1$. Now we set
    	$$\bold{Q}_{11}(X_0,\dots,X_{r-1}):=\det\begin{pmatrix}
    	\mathcal{L}_{21}&\dots &\mathcal{L}_{2(r-1)}\\
    	\vdots & & \vdots \\
    	\mathcal{L}_{r1}&\dots &\mathcal{L}_{r(r-1)}
    	\end{pmatrix}\in\overline{K}[X_0,\dots,X_{r-1}]$$
    	which is not identically zero as we assume that $Q_{11}\neq 0$ and $\bold{Q}_{11}(w_r,F^{\phi}_{\delta_{1}}(w_r),\dots,F^{\phi}_{\delta_{r-1}}(w_r))=Q_{11}$. Note that it is also a homogeneous polynomial of degree $r-1$. Similarly, we set
    	$$\bold{D}(X_0,\dots,X_{r-1}):=\det\begin{pmatrix}
    	\mathcal{L}_{10}&\dots &\mathcal{L}_{1(r-1)}\\
    	\vdots & & \vdots \\
    	\mathcal{L}_{r0}&\dots &\mathcal{L}_{r(r-1)}
    	\end{pmatrix}\in\overline{K}[X_0,\dots,X_{r-1}]$$
    	which is again not identically zero as we know by Proposition \ref{P:Anderson}(ii), \eqref{E:det4} and our assumption that \[
    	\bold{D}(w_r,F^{\phi}_{\delta_{1}}(w_r),\dots,F^{\phi}_{\delta_{r-1}}(w_r))=\tilde{\pi}\neq 0.
    	\]
    	 Furthermore $\bold{D}(X_0,\dots,X_{r-1})$ is a homogeneous polynomial of degree $r$. 
    	We prove the proposition by dividing our argument into several claims.
    	
    	\textbf{Claim 1: }  The set $\mathfrak{S}:=\big\{\frac{w_r}{\tilde{\pi}},\frac{F^{\phi}_{\delta_{1}}(w_r)}{\tilde{\pi}},\dots,\frac{F^{\phi}_{\delta_{r-1}}(w_r)}{\tilde{\pi}}\big\}$ is algebraically independent over $\overline{K}$.
    	\begin{proof}[{Proof of Claim 1}] Since 
    	\[
    	\overline{K}(w_r,F^{\phi}_{\delta_{1}}(w_r),\dots,F^{\phi}_{\delta_{r-1}}(w_r))=\overline{K}\Big(\frac{w_r}{\tilde{\pi}},\frac{F^{\phi}_{\delta_{1}}(w_r)}{\tilde{\pi}},\dots,\frac{F^{\phi}_{\delta_{r-1}}(w_r)}{\tilde{\pi}},\tilde{\pi}\Big)
    	\]
    	and $\trdeg_{\overline{K}}\overline{K}(P)=|\mathfrak{S}|=r$ by  \cite[Thm. 1.2.2]{CP12}, it suffices to show that $\tilde{\pi}$ is an algebraic element over the field 
    	$L:=\overline{K}\big(w_r/\tilde{\pi},F^{\phi}_{\delta_{1}}(w_r)/\tilde{\pi},\dots,F^{\phi}_{\delta_{r-1}}(w_r)/\tilde{\pi}\big).$
    	Let 
    	$$f_1(Y):=\bold{D}\Big(\frac{w_r}{\tilde{\pi}},\frac{F^{\phi}_{\delta_{1}}(w_r)}{\tilde{\pi}},\dots,\frac{F^{\phi}_{\delta_{r-1}}(w_r)}{\tilde{\pi}}\Big)Y^{r-1}-1=\left(\frac{Y}{\tilde{\pi}}\right)^{r-1}-1\in L[Y].$$ Then $f_1(Y)$ is a non-zero polynomial with coefficients in $L$ such that $f_1(\tilde{\pi})=0$, which implies that $\tilde{\pi}$ is algebraic over $L$.
    	\end{proof}
    	\textbf{Claim 2: } The element $\tilde{\pi}^{-2}w_rQ_{11}$ is algebraic over $L$. In particular, $$\trdeg_{\overline{K}}L=\trdeg_{\overline{K}}L(\tilde{\pi}^{-2}w_rQ_{11})=r.$$
    	
    	\begin{proof}[{Proof of Claim 2}] Let 
    	\begin{align*}
    	f_2(Y)&:=\bold{D}\Big(\frac{w_r}{\tilde{\pi}},\frac{F^{\phi}_{\delta_{1}}(w_r)}{\tilde{\pi}},\dots,\frac{F^{\phi}_{\delta_{r-1}}(w_r)}{\tilde{\pi}}\Big)^{r-2}Y^{r-1}\\
    	&\ \ \  \ \ \ \ \ \ \  \ \ \ \ \ \ \  \ \ \ \  -\left(\frac{w_r}{\tilde{\pi}}\bold{Q}_{11}\Big(\frac{w_r}{\tilde{\pi}},\frac{F^{\phi}_{\delta_{1}}(w_r)}{\tilde{\pi}},\dots,\frac{F^{\phi}_{\delta_{r-1}}(w_r)}{\tilde{\pi}}\Big)\right)^{r-1}\\
    	&=\left(\frac{Y}{\tilde{\pi}^{r-2}}\right)^{r-1}-\left(\frac{w_rQ_{11}}{\tilde{\pi}^r}\right)^{r-1}\in L[Y].
    	\end{align*}
    	Then $f_2(Y)$ is a non-zero polynomial with coefficients in $L$ such that $f_2(\tilde{\pi}^{-2}w_rQ_{11})=0$, which implies that $\tilde{\pi}^{-2}w_rQ_{11}$ is algebraic over $L$ as desired. The last assertion now follows from Claim 1. 
    	\end{proof}

    	\textbf{Claim 3: } The set $\mathfrak{S}$ is algebraically dependent over $\overline{K}(\tilde{\pi}^{-2}w_rQ_{11})$.
    	\begin{proof}[{Proof of Claim 3}]
    	
    	 Let
    	\begin{align*}
    	f_3(X_0,\dots,X_{r-1})&:=\bold{D}(X_0,X_1,\dots,X_{r-1})^{r-2}(\tilde{\pi}^{-2}w_rQ_{11})^{r-1}\\
    	&\ \ \ \ \ -\left(X_0\bold{Q}_{11}(X_0,X_1,\dots,X_{r-1})\right)^{r-1}\in\overline{K}(\tilde{\pi}^{-2}w_rQ_{11})[X_0,\dots,X_{r-1}].
    	\end{align*}
    	Since $\tilde{\pi}^{-2}w_rQ_{11}\neq 0$, $\bold{D}(X_0,X_1,\dots,X_{r-1})^{r-2}(\tilde{\pi}^{-2}w_rQ_{11})^{r-1}$ is a non-zero homogeneous polynomial of degree $r(r-2)$. On the other hand, $\left(X_0\bold{Q}_{11}(X_0,X_1,\dots,X_{r-1})\right)^{r-1}$ is a non-zero homogeneous polynomial of degree $r(r-1)$. Therefore $f_3$ is not identically zero as it is formed by two non-zero homogeneous polynomials of different total degree. Finally, since 
    	\[
    	f_3\Big(\frac{w_r}{\tilde{\pi}},\frac{F^{\phi}_{\delta_{1}}(w_r)}{\tilde{\pi}},\dots,\frac{F^{\phi}_{\delta_{r-1}}(w_r)}{\tilde{\pi}}\Big)=f_2(\tilde{\pi}^{-2}w_rQ_{11})=0,\]
    	 we conclude that $\mathfrak{S}$ is an algebraically dependent set over $\overline{K}(\tilde{\pi}^{-2}w_rQ_{11})$. 
    	\end{proof}
    	
    	 Now we prove the proposition. Assume    that $\tilde{\pi}^{-2}w_rQ_{11}$ is algebraic over $\overline{K}$. Then we have
    	\begin{align*}
    	r&=\trdeg_{\overline{K}}L(\tilde{\pi}^{-2}w_rQ_{11})\\
    	&=\trdeg_{\overline{K}}\overline{K}(\tilde{\pi}^{-2}w_rQ_{11})+\trdeg_{\overline{K}(\tilde{\pi}^{-2}w_rQ_{11})}L(\tilde{\pi}^{-2}w_rQ_{11})\\
  		&=\trdeg_{\overline{K}(\tilde{\pi}^{-2}w_rQ_{11})}L(\tilde{\pi}^{-2}w_rQ_{11})\\
  		&<r
    	\end{align*}
    	where the first equality follows from Claim 2, the second one from properties of the transcendental degree of field extensions, the third one from the assumption and the inequality from Claim 3. But it is a contradiction and hence we finish the proof of the proposition. 
    \end{proof}

     For $1\leq \ell\leq r-1$, we set $$J_\ell(\mathbf{z}):=\frac{g_\ell(\mathbf{z})^{(q^r-1)/(q^{\gcd(\ell,r)}-1)}}{\Delta_{r}(\zz)^{(q^\ell-1)/(q^{\gcd(\ell,r)}-1)}}$$
    which is a well-defined function on $\Omega^r$ as by construction $\Delta_{r}(\zz)$ is non-zero. It is known by \cite{Gek83} and \cite[Prop. 6.2]{Ham03} that if $\mathbf{z}$ is a CM point, then $J_\ell(\mathbf{z})\in\overline{K}$ for all $1\leq\ell\leq r$. 
    
    The main result of this section can be stated as follows. 
    	
    \begin{theorem}\label{T:trans}
    	Let $\mathbf{z}\in\Omega^r$ be a CM point. If $E_r(\mathbf{z})\neq 0$, then it is transcendental over $\overline{K}$.
    \end{theorem}

    \begin{proof}
         The proof uses Proposition \ref{Prop:Transcendence_of_Special_Values} and an idea of Chang \cite[Thm. 2.2.1]{Cha11}. Let us set $\mathbf{z}=(z_1\dots,z_r)^{\intercal}\in\Omega^r$ which, by assumption, is a CM point.
        We choose an element $w\in\mathbb{C}_\infty^{\times}$ so that $\Delta_r(\mathbf{z})w^{1-q^r}=(-1)^{r-1}$. Then by \cite[Prop. 6.2]{Ham03}, the Drinfeld module $$\phi_\theta^{w\zz}:=w\phi_\theta^\mathbf{z}w^{-1}=\theta+\epsilon^{1-q} J_1(\mathbf{z})^{\frac{q^{\gcd(1,r)}-1}{q^r-1}}\tau+\cdots+\epsilon^{1-q^{r-1}}J_{r-1}(\mathbf{z})^{\frac{q^{\gcd(r-1,r)}-1}{q^r-1}}\tau^{r-1}+(-1)^{r-1}\tau^r$$
        is defined over $\overline{K}$ with the period lattice $w\zz A$ where $\epsilon=(-1)^{\frac{1-q^r}{r-1}}$ is a $(1-q^{r})$-th power of a fixed $(r-1)$-st root of $-1$. Since  $\Delta_r(\mathbf{z})=w^{q^r-1}(-1)^{r-1}$, we have by \eqref{E:det6} that $h_r(\mathbf{z})= (w/\tilde{\pi})^{1+\dots+q^{r-1}}$.  By the uniqueness of the function $F^{\phi^{\zz}}_{\delta_{i}}$, we obtain
        $F_{\delta_{i}}^{\phi^{w\zz}}(X)=w^{q^i}F_{\delta_{i}}^{\phi^{\zz}}(X/w)$ for any  $X\in \CC_{\infty}$. In particular, we have $w^{q^i}F_{\delta_{i}}^{\phi^{\zz}}(z_j)=F_{\delta_{i}}^{\phi^{w\zz}}(wz_j)$ for any $1\leq i \leq r-1$ and $1\leq j \leq r$. Then Theorem~\ref{T:EGEK}(ii) implies that
        \begin{align*}
            E_r(\mathbf{z})&=\tilde{\pi}^{-1+q+\dots+q^{r-1}}h_r(\zz)\det\begin{pmatrix}
		    F^{\phi^{\zz}}_{\delta_{1}}(z_2)&\dots &F^{\phi^{\zz}}_{\delta_{r-1}}(z_2)\\
		    \vdots & & \vdots \\
		    F^{\phi^{\zz}}_{\delta_{1}}(z_r)&\dots &F^{\phi^{\zz}}_{\delta_{r-1}}(z_r)
		    \end{pmatrix}\\
		    &=\tilde{\pi}^{-1+q+\dots+q^{r-1}}\Big(\frac{w}{\tilde{\pi}}\Big)^{1+q+\cdots+q^{r-1}}\frac{1}{w^{q+\cdots q^{r-1}}}\det\begin{pmatrix}
		    F^{\phi^{w\zz}}_{\delta_{1}}(wz_2)&\dots &F^{\phi^{w\zz}}_{\delta_{r-1}}(wz_2)\\
		    \vdots & & \vdots \\
		    F^{\phi^{w\zz}}_{\delta_{1}}(wz_r)&\dots &F^{\phi^{w\zz}}_{\delta_{r-1}}(wz_r)
		    \end{pmatrix}\\
		    &= \frac{w}{\tilde{\pi}^2}\det\begin{pmatrix}
		    F^{\phi^{w\zz}}_{\delta_{1}}(wz_2)&\dots &F^{w\zz}_{\delta_{r-1}}(wz_2)\\
		    \vdots & & \vdots \\
		    F^{\phi^{w\zz}}_{\delta_{1}}(wz_r)&\dots &F^{\phi^{w\zz}}_{\delta_{r-1}}(wz_r)
		    \end{pmatrix}.
        \end{align*}
Since $wz_r=w$, we get the desired result by Proposition  \ref{Prop:Transcendence_of_Special_Values}.
    \end{proof}

\end{document}